%% file: main.tex
\documentclass[english]{article}
\usepackage[T1]{fontenc}
\usepackage[utf8]{inputenc}
\usepackage{geometry}
\usepackage{babel}
\usepackage{mathtools}
\usepackage{dsfont}
\usepackage{amsmath}
\usepackage{booktabs}
\usepackage{amssymb}
\usepackage[bookmarks=false,
 breaklinks=false,pdfborder={0 0 1},colorlinks=false]
 {hyperref}
\hypersetup{
 colorlinks,linkcolor=red,anchorcolor=blue,citecolor=blue}

\makeatletter
\usepackage{babel}
\usepackage{babel}
\usepackage{babel}

\usepackage{xcolor,colortbl}
\definecolor{Gray}{gray}{0.85}
\usepackage{tcolorbox}

\usepackage{mathtools}

\usepackage{cite}\usepackage{amsthm}\usepackage{dsfont}\usepackage{array}\usepackage{mathrsfs}\usepackage{comment}\onecolumn
\usepackage{natbib}
\usepackage{color}\usepackage{babel}

\newcommand{\R}{\mathbb{R}}

\allowdisplaybreaks

\usepackage{enumitem}
\setlist[itemize]{leftmargin=1.5em}
\setlist[enumerate]{leftmargin=1.5em}

\usepackage{babel}
\usepackage{algorithm}
\usepackage{algorithmic}
\usepackage{arydshln}

\numberwithin{equation}{section}

\newtheorem{theorem}{Theorem}[section]
\newtheorem{proposition}{Proposition}[section]
\newtheorem{definition}{Definition}[section]

\newtheorem{remark}{Remark}[section]

\newtheorem{assumption}{Assumption}[section]
\numberwithin{equation}{section}
\allowdisplaybreaks
\makeatother

\begin{document}
\title{Nesterov acceleration in optimizing over probability measures}
\author{
Jiaqi Tang\thanks{Department of Statistics, University of Wisconsin--Madison, Madison, WI 53706, USA; Email: \texttt{tang274@wisc.edu}.}
\and
Qin Li\thanks{Department of Mathematics, University of Wisconsin--Madison, Madison, WI 53706, USA; Email: \texttt{qinli@math.wisc.edu}.}
\and
Wilfrid Gangbo\thanks{Department of Mathematics, University of California, Los Angeles, Los Angeles, CA 90095, USA; Email: \\\texttt{wgangbo@math.ucla.edu}.}
}

\maketitle
\input{abstract.tex}


\input{intro.tex}\input{background.tex}\input{main_result.tex}\input{numerical_experiment.tex}

\section*{Acknowledgements}
This material is based upon work supported by the National Science Foundation under Grant No.~DMS-2424139, while all three authors visited the Simons Laufer Mathematical Sciences Institute in Berkeley, California, during the Fall 2025 semester to participate the semester program on ``Kinetic Theory: Novel Statistical, Stochastic and Analytical Methods". WG acknowledges the support by NSF grant DMS-2154578 and QL acknowledges the support by ONR Award AWD-100997 and Vilas Associate Award and JT acknowledges the support by NSF grant DMS-2308440.

\clearpage
\bibliographystyle{abbrvnat}
\bibliography{reference}

\end{document}

%% file: abstract.tex
\begin{abstract}
Optimization over probability measures has become an increasingly important paradigm in modern machine learning, scientific computing, and uncertainty quantification. Motivated by Nesterov's accelerated gradient method in Euclidean space, we develop Heavy-ball and Nesterov acceleration methods over the probability measure space $\mathcal{P}_2$ and establish non-asymptotic convergence guarantees that match their Euclidean counterparts. In particular, we derive convergence rates with respect to both the number of iterations and the number of particles used to represent the underlying probability distributions.

Extending accelerated optimization from Euclidean space to probability measures is challenging. The natural notion of momentum requires concepts such as tangent bundles of the set of probability space and they are hard to operate numerically. To overcome these difficulties, we introduce two complementary lifting procedures. The first lifts probability measures to phase space through a Hamiltonian formulation, introducing momentum variables into the dynamics. The second lifts probability measures to a common Hilbert space, restoring the linear structure required for convergence analysis while simultaneously yielding executable particle dynamics. Together, these two complementary lifting procedures provide a systematic methodology for designing, analyzing, and implementing momentum-based accelerated optimization methods over probability measure spaces.
\end{abstract}

%% file: intro.tex
\section{Introduction}\label{sec:intro}

Optimization over probability measures has attracted significant attention in recent years due to its broad applications in Bayesian inference~\citep{LW16, W18, LCBBR22, THS23, LSW23, HBD24, BPBMN25, CHHRS26}, uncertainty quantification~\citep{WCL22}, generative modeling~\citep{ACB17, C22, ABBGHLPSTT23, LSSVW25}, shallow neural network training~\citep{CB18, MMN18, SS20, DCLW22, CLTW25},  inverse problems~\citep{li2025inverseproblemsprobabilitymeasure,craig2026unfoldingwassersteinloss,vadeboncoeur2026distributional,wildey_butler}, and mean field games~\citep{Cardialaguet.et.al, GANGBO2022MR4499277, GANGBO2022MR4509653}. In the physical sciences, stationary states of complex systems can often be characterized as probability measures that minimize specific energy functionals as well, and this perspective is foundational in plasma physics~\citep{R01, R94, M81}, and classical density functional theory for fluids~\citep{EORK16}. 

For these problems, one seeks to solve
\begin{equation}\label{eqn:opt_rho}
    m_F=\min_{\rho\in\mathcal{P}_2(\mathbb{R}^d)}F[\rho]
\end{equation}
where $\mathcal{P}_2(\mathbb{R}^d)$ denotes the space of probability measures with finite second moment and $F$ is a functional that maps a probability measure to a real number. We denote $m_F$ the minimum value.

This problem can be viewed as an extension of classical optimization in Euclidean space,
\begin{equation}\label{eqn:opt_Rd}
    m_f=\min_{x\in\mathbb{R}^d}f(x)\,,
\end{equation}
where a vector in $\mathbb{R}^d$ that corresponds to the minimum value of the function $f$ is being sought. We denote $m_f$ the minimum value. For this finite-dimensional problem, we have a rich collection of efficient algorithms developed over the past decades. In particular, gradient-based methods and their accelerated variants have become foundational tools in modern optimization. It is therefore natural to ask whether these acceleration mechanisms can be extended to optimization over probability measures.

Among accelerated methods, Nesterov-type schemes play a central role due to their improved convergence rates over gradient descent. In Euclidean space, these methods incorporate a momentum term through updates of the form:
\begin{equation}
\begin{cases}
y_k = x_k + \beta_k (x_k-x_{k-1}),\\
x_{k+1} = y_k - s\nabla f(y_k)\,,
\end{cases}
\label{Ecl. Nes}
\end{equation}
which achieve accelerated convergence under suitable smoothness and convexity assumptions. A fundamental question motivating this work is:

\medskip
\begin{center}
\emph{Can one construct and analyze discrete-time, fully implementable Nesterov-type acceleration methods for optimization over probability measure spaces?}
\end{center}

The extension is highly nontrivial largely because the probability measure space $\mathcal{P}_2(\mathbb{R}^d)$ is inherently nonlinear. First and foremost, the gradient of a functional over $\mathcal{P}_2(\mathbb{R}^d)$ is not globally defined in a vector-space sense. While $\mathbb{R}^d$ is a linear space where $\nabla_xf$ is defined via the linear perturbation $f(x+\epsilon\Delta x)-f(x)$, a direct perturbation $\rho+\epsilon\delta\rho$ may not remain within $\mathcal{P}_2(\mathbb{R}^d)$, rendering the quantity $F[\rho+\epsilon\delta\rho]$ ill-defined.

Luckily, recent developments in optimal transport~\citep{villani2009optimal} provide a partial resolution to these difficulties by equipping $\mathcal{P}_2(\mathbb{R}^d)$ with the Wasserstein geometry~\citep{ambrosioGradientFlowsMetric2008}. This geometry enables the definition of a gradient locally via tangent spaces for each $\rho\in\mathcal{P}_2(\mathbb{R}^d)$, which in turn drives Wasserstein Gradient Flow (WGF) dynamics~\citep{JKO98}. 

However, acceleration requires substantially more structure than a gradient field alone. Classical Nesterov-type methods are fundamentally momentum-based: the update at each step depends not only on the current state, but also on the history of the trajectory through an associated velocity variable. While such momentum variables arise naturally in Euclidean space through differences of neighboring vector iterates ($x_k - x_{k-1}$), it is far less clear how to formulate analogous second-order inertia intrinsically on a curved, infinite-dimensional metric space.

A major conceptual breakthrough was recently achieved by~\citet{CLTW25}, who introduced a lifting from physical space to phase space to characterize a generalized Hamiltonian formulation. The central idea is to augment the position distribution with velocity variables, thereby lifting the dynamics into a phase-space distribution $\mu(t,x,v)$ whose spatial marginal recovers $\rho$. This construction successfully injects momentum into probability measure dynamics and achieves accelerated convergence rate. 

Just as in Euclidean setting where one needs to translate continuous-in-time Hamiltonian flow into a discrete-in-time algorithm that preserves convergence rate, here we face the same issue, and need to translate the flow into an algorithm that is supported by complexity guarantees. However, compared to optimization over $\R^d$, two new difficulties emerge:
\begin{itemize}
    \item \textbf{The Smoothness Bottleneck:} Convergence analysis of accelerated methods relies heavily on a \emph{global} smoothness condition (e.g., uniform Lipschitz continuity of gradients under a fixed linear geometry). In Wasserstein geometry, derivatives are defined relative to local tangent structures that vary across the space of probability measures, making global, uniform smoothness metrics difficult to establish.
    \item \textbf{The Realizability Bottleneck:} Probability measures are, by default, infinite-dimensional objects, and numerically one has to present it using a finite dimensional objects. A very practical way to do so is through Monte Carlo sampling and represent the underlying $\rho$ using its samples. Through this representation, the $\rho$ updates is translated to interactive particle updates, and we need to show the preservation of accuracy on the particle level.
\end{itemize}

To resolve these challenges, we introduce a second, critical layer of lifting, {which transforms} a probability measure to a random variable, and the associated optimization problem from the nonlinear space $\mathcal{P}_2(\mathbb{R}^d)$ into a Hilbert-space formulation over $L^2(\nu)$. By pulling the problem back to this flat Hilbert space, we restore a unified, linear geometry wherein uniform smoothness and Lipschitz conditions admit classic formulations. As a byproduct, this $L^2(\nu)$ lifting also provides an elegant, direct bridge between abstract measure updates and concrete particle trajectories, making the algorithm directly implementable.
\subsection{Main Contributions}

In this work, we build upon this dual-lifting architecture to design, analyze, and validate a fully discretized Nesterov acceleration scheme over $\mathcal{P}_2(\mathbb{R}^d)$. Our main contributions are summarized as follows:
\begin{enumerate}
    \item \textbf{Algorithm Synthesis via Dual Lifting:} We explicitly construct an accelerated iterate scheme over $\mathcal{P}_2(\mathbb{R}^d)$ by mapping the phase-space Hamiltonian dynamics on $\mathcal{P}_2(\mathbb{R}^d\times\mathbb{R}^d)$ onto the $L^2(\nu)$ Hilbert space of random variables. This yields an algorithmic update that is naturally practicalized via interacting particle systems.
    \item \textbf{End-to-End Non-Asymptotic Convergence Rates:} We establish a comprehensive convergence theory that simultaneously accounts for optimization error (in iteration), temporal discretization (step-size), and spatial discretization (particle size $N$). Crucially, we prove how these error terms decouple, providing explicit non-asymptotic bounds on the distance to the optimal solutions.
    \item \textbf{Theoretical and Numerical Verification:} We validate our framework across four structurally distinct machine learning and statistical settings: linear potentials under strict displacement convexity, Bayesian interacting particle sampling, shallow neural network training in the mean-field regime and nonlinear Fokker-Planck equation. Our numerical examples serve to verify that the empirical convergence behavior matches our derived non-asymptotic rates.
\end{enumerate}

The remainder of the paper is organized as follows. In Section~\ref{sec:background}, we review key concepts such as the accelerated optimization in Euclidean spaces, Wasserstein geometry, and the lifting procedure that maps $\mathcal{P}_2$ to $L^2$. Section~\ref{sec:main} presents the proposed algorithms together with their complete convergence theory, explicitly tracking dependencies on the step size and particle number. Quantitative numerical examples are presented in Section~\ref{sec:numerics} to validate our theoretical rates.

%% file: background.tex
\section{Background}\label{sec:background}

In this section, we review the geometric and analytical structures underlying the proposed accelerated methods. Section~\ref{sec:agd_Rd} recalls classical accelerated optimization algorithms in the Euclidean setting, emphasizing the role of momentum variables and smoothness assumptions in accelerated convergence theory. Section~\ref{sec:wasserstein} introduces basic notions from Wasserstein geometry that allow optimization problems to be formulated over probability measure spaces. Section~\ref{sec:lions_diff} presents the framework of lifting to the Hilbert-space that is central to our analysis and particle-level implementation.

\subsection{Accelerated Gradient Descent}\label{sec:agd_Rd}
We briefly review several classical accelerated optimization methods in Euclidean space that serve as templates for the constructions developed later in probability measure space. Consider the optimization problem
\begin{equation}\label{eqn:opt_euclidean_problem}
m_f=\min_{x\in\mathbb{R}^d} f(x),
\end{equation}
where $f:\mathbb{R}^d\to\mathbb{R}$ is continuously differentiable.

Recall that $f$ is called $m$-strongly convex ($m \ge 0$) if
\begin{equation}\label{eqn:ecl.convex.interp}
f\big((1-t)x+ty\big)
\le
(1-t)f(x)+t f(y)-\frac m2 t(1-t)\|x-y\|_2^2,
\qquad x,y\in\mathbb R^d,\quad t\in[0,1],
\end{equation}
and $L$-smooth ($L > 0$) if its gradient is $L$-Lipschitz continuous:
\begin{equation}
\|\nabla f(x)-\nabla f(y)\|_2
\le
L\|x-y\|_2,
\qquad x,y\in\mathbb R^d.
\label{ecl.lip}
\end{equation}
These two assumptions play a fundamental role in classical convergence analysis for first-order optimization methods.

The simplest first-order optimization method is gradient descent,
\begin{equation}
x_{k+1}=x_k-s\nabla f(x_k),
\label{Euc. gd}
\end{equation}
where $s>0$ denotes the step size. Under $L$-smoothness condition, gradient descent achieves the convergence rates
\begin{equation}\label{eqn:rate_GD}
\begin{cases}
f(x_k)-{m_f} \le \mathcal{O}(L/k), & \text{for convex } f \, (m=0),\\
f(x_k)-{m_f} \le \mathcal{O}((1-m/L)^k), & \text{for } m\text{-strongly convex } f \, (m>0),
\end{cases}
\end{equation}
when the step size is chosen as $s=1/L$.

Accelerated methods improve upon these rates by incorporating momentum information from previous iterates. One of the earliest momentum-based methods is the heavy-ball method~\citep{POLYAK19641},
\begin{equation}
x_{k+1} = x_k - s\nabla f(x_k) + \beta (x_k-x_{k-1}),
\label{Ecl.HB}
\end{equation}
where the difference term $x_k-x_{k-1}$ acts as a discrete velocity variable. Often it is convenient to introduce the velocity buffer $v_k:=(x_{k+1}-x_k)/\sqrt{s}$ and rewrite the heavy-ball iteration in a phase-space state form:
\begin{equation}\label{eqn:ecl_HB_xv}
\begin{cases}
x_k = x_{k-1} + \sqrt{s} v_{k-1}, \\
v_k = -\sqrt{s}\nabla f(x_k) + \beta v_{k-1}.
\end{cases}
\end{equation}

The appearance of the velocity term suggests an interpretation of accelerated optimization as a discrete approximation of a second-order dynamical system involving both position and velocity variables. Indeed, many accelerated methods admit continuous-time formulations through second-order ordinary differential equations and Hamiltonian dynamics. A prototypical example takes the form
\begin{equation}\label{eqn:second_order_ode}
\ddot{x}(t)+\gamma(t)\dot{x}(t)+\nabla f(x(t))=0,
\end{equation}
where the additional variable $v(t) = \dot{x}(t)$ tracks momentum. Such continuous-time perspective was systematically unified by \citet{wibisono2016variational}, who showed that a broad class of accelerated methods can be derived as the principle of least action applied to a systematic physical framework known as the Bregman Lagrangian.

While the heavy-ball method can exhibit faster convergence than gradient descent in certain strongly convex settings, accelerated convergence is not guaranteed for general convex objectives~\citep{Lessard}. For example, \citet{shi2018understandingaccelerationphenomenonhighresolution} showed that, by initializing with
\begin{equation}\label{ecl. HB. Ini}
x_1 = x_0-\frac{2s}{1+\sqrt{ms}}\nabla f(x_0)\,,
\end{equation}
and setting $s=\frac{m}{16L^2}$ and $\beta=\frac{1-\sqrt{ms}}{1+\sqrt{ms}}$, the iterates satisfy the convergence estimate
\begin{equation}
f(x_k)-{m_f} \le \mathcal{O}\!\left( \left(1+\frac{m}{16L}\right)^{-k} \right)\,.
\end{equation}
This rate is comparable to gradient descent~\eqref{eqn:rate_GD}.

Nesterov proposed a modified momentum scheme in which the gradient is evaluated at a look-ahead point~\citep{nesterov2013introductory}:
\begin{equation}\label{Ecl. Nes}
\begin{cases}
y_k = x_k + \beta_k (x_k-x_{k-1}),\\
x_{k+1} = y_k - s\nabla f(y_k).
\end{cases}
\end{equation}
Equivalently, the map from $(x_{k-1},v_{k-1})$ to $(x_k,v_k)$ can formulated as:
\begin{equation}\label{eqn:ecl_Nes_yvx}
\begin{cases}
y_k = x_{k-1}+\sqrt{s}(1+\beta_k) v_{k-1}\,,\\
x_k = y_k-\sqrt{s}\beta_k v_{k-1}\,,\\
v_k = -\sqrt{s}\nabla f(y_k)+\beta_k v_{k-1}\,.
\end{cases}
\end{equation}
Note that $y_k$ is a look-ahead point, and serves as an auxiliary variable. It does not enter the $(x_{k-1},v_{k-1})\to (x_{k},v_k)$ map.

With suitable choices of $s$ and $\beta_k$, this method achieves the celebrated accelerated convergence rates
\begin{equation}
\begin{cases}
f(x_k)-{m_f} \le \mathcal{O}(L/k^2), & \text{for convex } f,\\
f(x_k)-{m_f} \le \mathcal{O}((1-\sqrt{m/L})^k), & \text{for } m\text{-strongly convex } f.
\end{cases}
\label{ecl.nes.rate}
\end{equation}
The choice of coefficients requires precise tuning, summarized as follows:
\begin{itemize}
    \item \textbf{Convex Case ($m=0$):} $s=\frac{1}{L}, \quad \kappa_0=1, \quad \kappa_{k+1}=\frac{1+\sqrt{1+4\kappa_{k}^2}}{2}, \quad \beta_k=\frac{\kappa_{k}-1}{\kappa_{k+1}}$.
    Throughout this paper we term is Nesterov-c.
    \item \textbf{Strongly Convex Case ($m>0$):} $s=\frac{1}{L}, \quad \beta_k = \beta = \frac{1-\sqrt{ms}}{1+\sqrt{ms}}$.
    Throughout this paper we term is Nesterov-sc.
\end{itemize}

Table~\ref{tab:euclidean-methods} summarizes the momentum structures and convergence rates of these classical first-order optimization methods in Euclidean space. 

\begin{table}[ht]
\centering
\caption{Momentum structures and convergence rates of classical first-order optimization methods in Euclidean space.}
\label{tab:euclidean-methods}
\vspace{0.2cm}
\begin{tabular}{lcc}
\hline
\textbf{Method} & \textbf{Convex} ($m=0$) & \textbf{Strongly Convex} ($m>0$) \\ \hline
Gradient Descent & $\mathcal{O}(L/k)$ & $\mathcal{O}((1-m/L)^k)$ \\ 
Heavy-Ball & No general guarantee & $\mathcal{O}\!\left(\left(1+\frac{m}{16L}\right)^{-k}\right)$ \\ 
Nesterov (-c or -sc) & $\mathcal{O}(L/k^2)$ & $\mathcal{O}((1-\sqrt{m/L})^k)$ \\ \hline
\end{tabular}
\end{table}

A common feature of these accelerated methods is their fundamental reliance on vector subtraction to construct momentum buffers, combined with global smoothness estimates under a fixed linear geometry. Extending these structures to probability measure spaces is substantially more delicate and requires additional geometric constructions, which we review next.

\subsection{Basics of Wasserstein Space}\label{sec:wasserstein}

We review several key notions from Wasserstein geometry that provide the variational framework for optimization over probability measures. Wasserstein geometry equips $\mathcal{P}_2(\mathbb{R}^d)$ with concepts of distance, gradient, and convexity analogous to those in Euclidean optimization. These structures form the geometric foundation underlying accelerated dynamics in probability measure spaces.

We begin with two basic notions. Let $\mathcal P_2(\mathbb R^d)$ denote the set of probability measures on $\mathbb R^d$ with finite second moment. For a measurable map $T:\mathbb{R}^n\to\mathbb{R}^m$ and a measure $\nu\in\mathcal{P}_2(\mathbb{R}^n)$, the pushforward measure $T_{\#}\nu \in \mathcal{P}_2(\mathbb{R}^m)$ is defined by
\begin{equation}
(T_{\#}\nu)(A) := \nu\bigl(T^{-1}(A)\bigr), \qquad \text{for any Borel set } A \subset \mathbb{R}^m,
\label{pushforward}
\end{equation}
where $T^{-1}(A) := \{x \in \mathbb{R}^n : T(x) \in A\}$. The space $\mathcal{P}_2(\mathbb{R}^d)$ becomes a complete metric space when endowed with the Wasserstein-2 distance.

\begin{definition}[Wasserstein-$2$ distance]
Let $\rho_0,\rho_1\in\mathcal P_2(\mathbb R^d)$. The Wasserstein-$2$ distance between $\rho_0$ and $\rho_1$ is defined as
\begin{equation}
W_2^2(\rho_0,\rho_1)
=
\inf_{\pi\in\Pi(\rho_0,\rho_1)}
\int_{\mathbb R^d\times\mathbb R^d}
\Vert x-y\Vert_2^2\,d\pi(x,y),
\label{W_2}
\end{equation}
where $\Pi(\rho_0,\rho_1)$ denotes the set of probability measures on $\mathbb{R}^d\times\mathbb{R}^d$ whose first and second marginals are $\rho_0$ and $\rho_1$, respectively. Elements of $\Pi(\rho_0,\rho_1)$ are called couplings between $\rho_0$ and $\rho_1$. We denote by $\Pi_0(\rho_0,\rho_1)\subseteq \Pi(\rho_0,\rho_1)$ the set of optimal couplings attaining the infimum in \eqref{W_2}.
\end{definition}
The set of optimal couplings $\Pi_0(\rho_0, \rho_1)$ is guaranteed to be nonempty~\citep{villani2009optimal}. The metric space $(\mathcal{P}_2(\mathbb{R}^d), W_2)$ provides a natural geometric setting for defining gradients and variational flows of functionals on probability measures. We define the effective domain of a functional $F: \mathcal{P}_2(\mathbb{R}^d) \to \mathbb{R} \cup \{+\infty\}$ as the set
\begin{equation}\label{eqn:def_domain_F}
D(F) := \bigl\{\rho \in \mathcal{P}_2(\mathbb{R}^d) : F[\rho] < +\infty \bigr\}.
\end{equation}
\begin{definition}[Wasserstein gradient]
\label{Wasserstein gradient}
Let $F:\mathcal{P}_2(\mathbb{R}^d) \to \mathbb{R} \cup \{+\infty\}$. We say that $F$ is \emph{Wasserstein differentiable} at $\rho \in D(F)$ if there exists a vector field $\xi \in L^2(\rho)$, where
\begin{equation}
L^2(\rho) := \left\{ \xi : \mathbb{R}^d \to \mathbb{R}^d : \int_{\mathbb{R}^d} \|\xi(x)\|_2^2 \, d\rho(x) < \infty \right\},
\end{equation}
such that for every $\eta \in \mathcal{P}_2(\mathbb{R}^d)$ and some optimal coupling $\gamma \in \Pi_0(\rho, \nu)$,
\begin{equation}\label{eq:W_gradient}
F[\nu] - F[\rho] = \int_{\mathbb{R}^d \times \mathbb{R}^d} \xi(x) \cdot (y - x) \, d\gamma(x, y) + {o}\bigl(W_2(\rho, \eta)\bigr).
\end{equation}
In this case, the vector field $\xi$ that attains the minimal $L^2(\rho)$-norm among all vector fields satisfying~\eqref{eq:W_gradient} is called the \emph{Wasserstein gradient} of $F$ at $\rho$, and is denoted by $\nabla_{W_2}F[\rho]$.
\end{definition}

This definition plays a role analogous to the Fr\'echet derivative in the metric space $(\mathcal{P}_2(\mathbb{R}^d), W_2)$. The linear expansion characterizes the first-order variation of the functional $F$ along optimal transport directions. Alternatively, the Wasserstein gradient can be expressed through the first variation.

\begin{definition}[First variation]
Let $F:\mathcal{P}_2(\mathbb{R}^d) \to \mathbb{R} \cup \{+\infty\}$ and let $\rho \in D(F)$. A function $\frac{\delta F}{\delta\rho}[\rho]: \mathbb{R}^d \to \mathbb{R}$ is called {a} \emph{first variation}\footnote{Note that whenever a first variation $\frac{\delta F}{\delta\rho}[\rho]$ exists, then for any $\lambda \in \mathbb R$, $\frac{\delta F}{\delta\rho}[\rho]+\lambda$ is also a first variation.} of $F$ at $\rho$ if for any $\nu \in \mathcal{P}_2(\mathbb{R}^d)$,
\begin{equation}
\left. \frac{d}{d\varepsilon} F\bigl[(1-\varepsilon)\rho + \varepsilon\nu\bigr] \right|_{\varepsilon=0} = \int_{\mathbb{R}^d} \frac{\delta F}{\delta\rho}[\rho](x) \, d(\nu - \rho)(x).
\end{equation}
\end{definition}

If we further assume that $\frac{\delta F}{\delta\rho}[\rho]$ is sufficiently regular, then
\begin{equation}
\nabla_{W_2}F[\rho](x) = \nabla_x \frac{\delta F}{\delta\rho}[\rho](x),
\end{equation}
where $\nabla_x$ denotes the classical Euclidean gradient with respect to $x$~\citep{ambrosioGradientFlowsMetric2008}.

\begin{remark}[Examples of Wasserstein gradients]
We recall three canonical examples of functionals and their corresponding Wasserstein gradients:
\begin{itemize}
    \item \textbf{Potential Energy:} For $F[\rho] = \int_{\mathbb{R}^d} h(x) \, d\rho(x)$, we have $\nabla_{W_2}F[\rho](x) = \nabla_x h(x)$.
    \item \textbf{Internal Energy:} For $F[\rho] = \int_{\mathbb{R}^d} U(\rho(x)) \, dx$, where $\rho(x)$ is the Radon-Nikodym density, we have $\nabla_{W_2}F[\rho](x) = \nabla_x\bigl(U'(\rho(x))\bigr)$.
    \item \textbf{Interaction Energy:} For $F[\rho] = \frac{1}{2} \int_{\mathbb{R}^d \times \mathbb{R}^d} K(x-y) \, d\rho(x) d\rho(y)$, we have $\nabla_{W_2}F[\rho](x) = \nabla_x(K * \rho)(x)$.
\end{itemize}
\end{remark}

Analogous to convexity in Euclidean spaces, convexity for functionals over $(\mathcal{P}_2(\mathbb{R}^d), W_2)$ is defined along geodesic interpolations.

\begin{definition}[Geodesic convexity]
Let $F:\mathcal{P}_2(\mathbb{R}^d) \to \mathbb{R} \cup \{+\infty\}$. For $m \ge 0$, $F$ is called $m$-geodesically convex if for any $\rho_0, \rho_1 \in D(F)$ and any optimal coupling $\gamma \in \Pi_0(\rho_0, \rho_1)$,
\begin{equation}
F[\rho_t] \le (1-t)F[\rho_0] + tF[\rho_1] - \frac{m}{2}t(1-t)W_2^2(\rho_0, \rho_1), \qquad \forall t \in [0, 1],
\label{W-interp-convex}
\end{equation}
where $\rho_t$ is the geodesic interpolation measure defined by $\rho_t := (T_t)_{\#}\gamma$, with $T_t(x, y) := (1-t)x + ty$.
\end{definition}

When $F$ is Wasserstein differentiable, geodesic convexity admits a standard first-order characterization.

\begin{proposition}[First-order characterization of geodesic convexity]
Let $F:\mathcal{P}_2(\mathbb{R}^d) \to \mathbb{R} \cup \{+\infty\}$ be Wasserstein differentiable on $D(F)$. Then $F$ is $m$-geodesically convex if and only if for any $\rho_0, \rho_1 \in D(F)$ and any $\gamma \in \Pi_0(\rho_0, \rho_1)$,
\begin{equation}
F[\rho_1] - F[\rho_0] \ge \int_{\mathbb{R}^d \times \mathbb{R}^d} \nabla_{W_2}F[\rho_0](x) \cdot (y - x) \, d\gamma(x, y) + \frac{m}{2}W_2^2(\rho_0, \rho_1).
\end{equation}
\end{proposition}

\begin{remark}[Examples of geodesically convex functionals]
If $h:\mathbb{R}^d \to \mathbb{R}$ is $m$-strongly convex, the potential energy $F[\rho] = \int_{\mathbb{R}^d} h(x) \, d\rho(x)$ is $m$-geodesically convex. Similarly, the relative entropy $F[\rho] = \int_{\mathbb{R}^d} \rho(x) \log \frac{\rho(x)}{\rho^*(x)} \, dx$ is $m$-geodesically convex if the reference distribution $\rho^*(x) \propto \exp(-h(x))$ for an $m$-strongly convex potential $h$.
\end{remark}

\begin{remark}[The Smoothness Bottleneck]
\label{remark. l smooth}
According to Definition~\ref{Wasserstein gradient}, the Wasserstein gradient $\nabla_{W_2}F[\rho]$ intrinsically belongs to the Hilbert space $L^2(\rho)$. Because this space changes according to the change of the underlying measure $\rho$, the gradients $\nabla_{W_2}F[\rho_0]$ and $\nabla_{W_2}F[\rho_1]$ are defined over entirely different spaces when $\rho_0 \neq \rho_1$. Consequently, formulating a global Lipschitz or smoothness condition analogous to the Euclidean requirement~\eqref{ecl.lip} seems impossible in this Wasserstein setting. This lack of a unified linear geometry poses the primary mathematical obstacle, motivating us to lift these objects to Hilbert spaces. 
\end{remark}
\subsection{Hilbert Space Lifting of Probability Measures}\label{sec:lions_diff}

As discussed in Remark~\ref{remark. l smooth}, Wasserstein gradients at different probability measures naturally belong to different tangent spaces. Consequently, unlike the Euclidean setting, there is no single ambient Hilbert space in which gradients can be globally compared. This lack of a unified linear structure presents a major obstacle for defining global smoothness conditions analogous to \eqref{ecl.lip}, thereby complicating both stable discrete-time discretization and accelerated complexity analysis.

To overcome this difficulty, we introduce a Hilbert-space lifting of probability measures. Let $(\Omega, \nu)$ be a non atomic probability space (for example, $\Omega=(0,1)^d$ and $\nu$ is Lebesgue measure, or  $\Omega=\mathbb{R}^d$ and $\nu$ is a standard Gaussian). This $\nu$ will be regarded as our reference probability measure. We consider the Hilbert space
\begin{equation}
L^2(\nu) := \left\{ X : \Omega \to \mathbb{R}^d : \int_{\mathbb{R}^d} \|X(z)\|_2^2 \, d\nu(z) < +\infty \right\}.
\end{equation}
We equip $L^2(\nu)$ with the standard inner product and its induced norm:
\begin{equation}
\langle f, g \rangle_\nu := \int_{\Omega} \langle f(z), g(z) \rangle \, d\nu(z), \qquad \|f\|_\nu^2 := \int_{\Omega} \|f(z)\|_2^2 \, d\nu(z),
\end{equation}
for $f, g \in L^2(\nu)$.

The key connection between $\mathcal{P}_2(\mathbb{R}^d)$ and $L^2(\nu)$ is provided by the pushforward operation. Any map $X \in L^2(\nu)$ induces a probability measure $X_{\#}\nu \in \mathcal{P}_2(\mathbb{R}^d)$.  Conversely, since $\nu$ is non-atomic, every probability measure $\rho \in \mathcal{P}_2(\mathbb{R}^d)$ admits a representation of the form $\rho = X_{\#}\nu$ for some measurable map $X \in L^2(\nu)$~\citep{villani2009optimal}. Although such representations are generally non-unique, they embed probability measures into a unified Hilbert-space framework and allow variational structures on $\mathcal{P}_2(\mathbb{R}^d)$ to be analyzed through the flat linear geometry of $L^2(\nu)$.

Through this correspondence, a functional $F$ can be lifted to a functional $\widetilde{F}$ on the Hilbert space according to:
\begin{equation}\label{eqn:lifts}
F: \mathcal{P}_2(\mathbb{R}^d) \to \mathbb{R} \cup \{+\infty\} \quad \Rightarrow \quad \widetilde{F}: L^2(\nu) \to \mathbb{R} \cup \{+\infty\}, \quad \text{where } \widetilde{F}(X) := F[X_{\#}\nu].
\end{equation}
Since $L^2(\nu)$ is a standard Hilbert space, the gradient of $\widetilde{F}$ can be defined in the classical Fr\'echet sense, denoted by $\nabla_{L^2}\widetilde{F}(X) \in L^2(\nu)$. In contrast with the Wasserstein setting, this lifted formulation admits a natural global notion of smoothness because all gradients now live in the same ambient space $L^2(\nu)$.

It is easy to see that such a lift is not unique: there may exist $X_1\neq X_2$ such that $X_{1\#}\nu=X_{2\#}\nu$. This suggests the probability measure space $\mathcal P_2(\mathbb R^d)$ can be viewed as a quotient of $L^2(\Omega,\nu;\mathbb R^d)$ under the equivalence relation $X_1\sim X_2$ if $X_{1\#}\nu=X_{2\#}\nu$. A finite-dimensional analogue illustrates this viewpoint. Let \(\Omega=(0,1)^d\), and divide each coordinate interval \((0,1)\) into $N$ subintervals of equal length. This partitions \(\Omega\) into $N^d$ cubes $C_1,\ldots,C_{N^d}$ of equal measure. Every permutation $p$ of these $N^d$ cubes induces a measure-preserving map $s_p:\Omega\to\Omega$ that maps each cube $\{C_i\}$ to $\{C_{p(i)}\}$. Let $\mathcal S_N=\{s_p:p\in S_{N^d}\}$ denote the resulting group of measure-preserving rearrangements. Let \(L_N^2(\Omega;\mathbb R^d)\) be the space of functions that are constant on each cube. Then the quotient space $L_N^2(\Omega;\mathbb R^d)/\mathcal S_N$ is naturally identified with the set of equally weighted empirical measures $\mathcal P_N = \left\{\frac{1}{N^d}\sum_{i=1}^{N^d}\delta_{x_i}: x_i\in\mathbb R^d \right\}$.\footnote{Indeed, every \(X\in L_N^2(\Omega;\mathbb R^d)\) can be written as $X(\omega)=\sum_{i=1}^{N^d}x_i\mathsf 1_{C_i}(\omega)$ and its law is $X_{\#}\nu=\frac{1}{N^d}\sum_{i=1}^{N^d}\delta_{x_i}$. Composing \(X\) with $s_p$ leaves its law unchanged. Conversely, two functions have the same law if and only if their cell values agree up to a permutation.} Passing $N\to\infty$ limit, the quotient space becomes $L^2_\infty(\Omega)/\mathcal S_\infty=L^2(\Omega, \mathbb{R}^d)/\mathcal{S}_\infty$ that is $\mathcal{P}_2(\mathbb{R}^d)$ endowed with the Wasserstein metric.

A fundamental question is whether this lifted formulation preserves the geometric structures underlying optimization on $\mathcal{P}_2(\mathbb{R}^d)$. To establish this consistency, we introduce the following structural assumptions on the underlying functional:

\begin{assumption}\label{assum:regularity}
We assume the functional $F$ satisfies the following regularity conditions:
\begin{enumerate}[leftmargin=4em]
    \item[\normalfont\bfseries{(A1)}] The functional $F$ is globally finite on $\mathcal{P}_2(\mathbb{R}^d)$, i.e., $D(F) = \mathcal{P}_2(\mathbb{R}^d)$.
    \item[\normalfont\bfseries{(A2)}] $F$ is Wasserstein differentiable at every $\rho \in \mathcal{P}_2(\mathbb{R}^d)$.
\end{enumerate}
\end{assumption}

\begin{remark}[Scope and Boundaries of Assumption~\ref{assum:regularity}]
Assumption~\ref{assum:regularity} is satisfied by a broad class of functionals arising in mean-field optimization and interacting particle systems:
\begin{itemize}[leftmargin=2.5em]
    \item \textbf{Linear Potentials:} Any functional $F[\rho] = \int_{\mathbb{R}^d} h(x) \, d\rho(x)$ satisfies the assumption provided $h$ is differentiable with at most linear gradient growth.
    \item \textbf{Interaction Energies:} $F[\rho] = \frac{1}{2} \int_{\mathbb{R}^d \times \mathbb{R}^d} K(x-y) \, d\rho(x) d\rho(y)$ qualifies if the interaction kernel $K$ is continuously differentiable and $\nabla K$ has at most linear growth.
    \item \textbf{Mean-Field Content:} Functionals of the form $F[\rho] = \int_{\mathbb{R}^d} g(x, \rho) \, d\rho(x)$ satisfy the conditions if $g(x, \rho)$ is differentiable in $x$ and Wasserstein differentiable in $\rho$, with an $x$-derivative exhibiting at most linear growth~\citep{erny2025first}.
    \item \textbf{The Distance Counter-Example:} For a fixed measure $\mu \in \mathcal{P}_2(\mathbb{R}^d)$, the functional $F[\rho] = -\frac{1}{2} W_2^2(\rho, \mu)$ is finite and continuous, yet it fails to be Wasserstein differentiable everywhere. This unique behavior highlights the non-linearities intrinsic to Riemannian-type geometries.
\end{itemize}
It is crucial to note that Assumption~\ref{assum:regularity} does not directly cover exact entropy-type functionals (e.g., the standard Kullback-Leibler (KL) divergence) because they blow up on measures that are not absolutely continuous. For our Bayesian sampling application in Section~\ref{sec:numerics}, we circumvent this by adopting a regularized version of KL divergence.
\end{remark}

\begin{proposition}[\cite{GANGBO2019119}]\label{prop2.1}
Let $X \in L^2(\nu)$ and let $F$ and $\widetilde{F}$ be defined as in~\eqref{eqn:lifts}. Suppose that condition \textbf{(A1)} holds. Then $F$ is Wasserstein differentiable at $\rho = X_{\#}\nu$ if and only if $\widetilde{F}$ is Fr\'echet differentiable at $X \in L^2(\nu)$. Moreover, the gradients are coupled via the composition:
\begin{equation}\label{identity}
\nabla_{L^2}\widetilde{F}(X) = \nabla_{W_2}F[\rho] \circ X.
\end{equation}
\end{proposition}

The identity \eqref{identity} shows that abstract Wasserstein gradients can be cleanly represented as ordinary Hilbert-space gradients after lifting. We now map the two structural properties underlying accelerated optimization: convexity and smoothness.

\noindent\textbf{Convexity Preservations:}
We say that the lifted functional $\widetilde{F}$ is $m$-convex if, for any $X, Y \in L^2(\nu)$ and $t \in [0, 1]$:
\begin{equation}\label{L2 convexity interp}
\widetilde{F}\bigl((1-t)X + tY\bigr) \le (1-t)\widetilde{F}(X) + t\widetilde{F}(Y) - \frac{m}{2}t(1-t)\|X - Y\|_\nu^2.
\end{equation}
When $\widetilde{F}$ is Fr\'echet differentiable, this is equivalent to the first-order condition:
\begin{equation}\label{L2 convexity}
\widetilde{F}(Y) - \widetilde{F}(X) \ge \langle \nabla_{L^2}\widetilde{F}(X), Y - X \rangle_\nu + \frac{m}{2}\|X - Y\|_\nu^2.
\end{equation}

\begin{proposition}[\cite{GANGBO2022MR4509653}]\label{prop2.2}
Let $F$ and $\widetilde{F}$ be defined according to~\eqref{eqn:lifts}, and assume $F$ satisfies Assumption~\ref{assum:regularity}. Then $F$ is $m$-geodesically convex on $\mathcal{P}_2(\mathbb{R}^d)$ if and only if $\widetilde{F}$ is $m$-convex on $L^2(\nu)$.
\end{proposition}

\textit{We note that this correspondence holds under more relaxed conditions. For instance, when $F$ is continuous on $\mathcal{P}_2(\mathbb{R}^d)$ for dimensions $d \ge 2$, geodesic convexity can still be passed directly to classical convexity for $\widetilde{F}$ even when full differentiability is not provided.}

\noindent\textbf{Smoothness Preservations:}
We say that the lifted functional $\widetilde{F}$ is $L$-smooth if its Fr\'echet gradient is $L$-Lipschitz continuous:
\begin{equation}\label{L2 lip}
\|\nabla_{L^2}\widetilde{F}(X) - \nabla_{L^2}\widetilde{F}(Y)\|_\nu \le L\|X - Y\|_\nu, \qquad \forall X, Y \in L^2(\nu).
\end{equation}
The following proposition guarantees that this Hilbert-space smoothness is directly inherited from the structural regularity of the underlying Wasserstein gradient. {It  can be obtained by combining Remark 2.12 and Lemma 2.13 in~\cite{GANGBO2022MR4509653}. and we cite it here:}

\begin{proposition}\label{L2 lip cond}
Let $F$ and $\widetilde{F}$ be defined as in~\eqref{eqn:lifts}. Suppose that $F$ satisfies Assumption~\ref{assum:regularity} and that there exists a global constant $L > 0$ such that:
\begin{equation}\label{eqn:F_diff_L_smooth}
\|\nabla_{W_2}F[\rho_0](x) - \nabla_{W_2}F[\rho_1](y)\|_2 \le \frac{L}{2} \Bigl( \|x - y\|_2 + W_2(\rho_0, \rho_1) \Bigr),
\end{equation}
for all $\rho_0, \rho_1 \in \mathcal{P}_2(\mathbb{R}^d)$ and $x, y \in \mathbb{R}^d$. Then $\widetilde{F}$ is $L$-smooth on $L^2(\nu)$ in the sense of \eqref{L2 lip}.
\end{proposition}

These structural equivalences are very powerful. They demonstrate that the Hilbert-space lifting simultaneously resolves the primary analytical and computational barriers of Wasserstein optimization. Analytically, it restores a uniform linear geometry where convexity and smoothness match their classical Euclidean formulations. Computationally, by representing probability measures via random variables in a flat $L^2(\nu)$ space, it allows measure-valued optimization equations to be directly parameterized and executed via concrete particle systems.

%% file: main_result.tex
\section{Main Results}\label{sec:main}

We now present the main results of the paper. In Section~\ref{sec:algorithms}, we construct accelerated optimization methods over probability measures based on the dual-lifting framework. Section~\ref{sec:convergence_theorem} establishes their end-to-end convergence properties together with quantitative non-asymptotic complexity estimates.

\subsection{Algorithms}\label{sec:algorithms}

We construct accelerated optimization methods over probability measures by combining the phase-space and Hilbert-space lifting procedures introduced in the previous sections. Following the Hamiltonian perspective discussed by \citet{wibisono2016variational}, the optimization dynamics are formulated on the phase space $\mathbb{R}^d \times \mathbb{R}^d$, whose individual components track position and velocity variables, respectively. The resulting joint phase-space measures encode both spatial and momentum information, providing probabilistic analogues of classical momentum-based acceleration methods in Euclidean optimization.

{Denote by $\mathsf{P}_x, \mathsf{P}_v: \mathbb R^d \times \mathbb R^d \to \mathbb R^d$ the projection operator projecting to the first and second variable. We consider a sequence of joint probability measures $\{\mu_k\}_{k\ge0} \subset \mathcal{P}_2(\mathbb{R}^d \times \mathbb{R}^d)$, and denote by $\mu_k^X$ and $\mu_k^V$ their respective spatial and velocity marginals:
\begin{equation}
\mu_k^X:=\mathsf{P}_{x\, \#}\mu^k, \qquad \mu_k^V  := \mathsf{P}_{v\, \#} \mu^k.
\end{equation}
}
The primary optimization iterate is identified with the position marginal:
\begin{equation}
\rho_k := \mu_k^X \in \mathcal{P}_2(\mathbb{R}^d).
\end{equation}

This phase-space formulation allows Euclidean momentum methods to be transferred into the probability measure space via pushforward maps acting directly on the joint distributions.

\noindent\textbf{Heavy-Ball Acceleration Over Measures:} 
Recall the Euclidean heavy-ball iteration~\eqref{eqn:ecl_HB_xv}. We define the deterministic drift transport map $T^{\mathrm{drift}}: \mathbb{R}^d \times \mathbb{R}^d \to \mathbb{R}^d \times \mathbb{R}^d$ by
\begin{equation}
T^{\mathrm{drift}}(x, v) := \bigl(x + \sqrt{s}\,v, \, v\bigr),
\end{equation}
and the position-dependent acceleration map $T_k^{\mathrm{HB}}: \mathbb{R}^d \times \mathbb{R}^d \to \mathbb{R}^d \times \mathbb{R}^d$ by
\begin{equation}
T_k^{\mathrm{HB}}(x, v) := \left( x, \, -\sqrt{s}\,\nabla_{W_2}F[\hat{\rho}_k](x) + \beta v \right), \quad \text{where } \hat{\rho}_k := \bigl(T^{\mathrm{drift}}\bigr)_{\#}\mu_{k-1}\Bigr|^X.
\end{equation}
The corresponding continuous-measure phase-space heavy-ball iteration is given by
\begin{equation}\label{HB}
\begin{cases}
\hat{\mu}_k = (T^{\mathrm{drift}})_{\#}\mu_{k-1}, \\[0.3em]
\mu_k = (T_k^{\mathrm{HB}})_{\#}\hat{\mu}_k.
\end{cases}
\end{equation}
The first step transports the position coordinate according to the current velocity field, while the second updates the velocity through a combination of the local Wasserstein gradient field and the inertia parameter $\beta$. Following the Euclidean initialization rule~\eqref{ecl. HB. Ini}, the joint distribution is initialized via:
\begin{equation}\label{eqn:HB_initial}
\mu_0 = \left( \text{Id}, \; -\frac{2\sqrt{s}}{1+\sqrt{ms}}\,\nabla_{W_2}F[\rho_0] \right)_{\#}\rho_0\,.
\end{equation}

While the phase-space updates \eqref{HB} are specified at the abstract level of probability measures, the Hilbert-space lifting developed in Section~\ref{sec:lions_diff} allows these dynamics to be directly practicalized using Monte Carlo particle tracking. We approximate the continuous phase-space measures by the empirical distributions of $N$ interacting particles. Drawing initial states $\{(x_0^{(i)}, v_0^{(i)})\}_{i=1}^N \sim \mu_0$, we form the empirical approximations:
\begin{equation}
\overline{\mu}_0 = \frac{1}{N}\sum_{i=1}^N \delta_{(x_0^{(i)}, v_0^{(i)})}, \qquad \overline{\rho}_0 = \frac{1}{N}\sum_{i=1}^N \delta_{x_0^{(i)}},
\end{equation}
where, in accordance with \eqref{eqn:HB_initial}, the initial velocities are set to
\begin{equation}\label{eqn:HB_initial_particle}
x_0^{(i)} \overset{\mathrm{i.i.d.}}{\sim} \rho_0, \qquad v_0^{(i)} = -\frac{2\sqrt{s}}{1+\sqrt{ms}}\nabla_{W_2}F[\overline{\rho}_0](x_0^{(i)}).
\end{equation}

Applying the pushforward updates \eqref{HB} to these discrete empirical particles yields the interacting heavy-ball particle system:
\begin{equation}\label{eqn:particle_HB}
\begin{cases}
x_k^{(i)} = x_{k-1}^{(i)} + \sqrt{s}\,v_{k-1}^{(i)}, \\[0.5em]
v_k^{(i)} = -\sqrt{s}\,\nabla_{W_2}F[\overline{\rho}_k](x_k^{(i)}) + \beta v_{k-1}^{(i)},
\end{cases}
\end{equation}
where $\overline{\rho}_k = \frac{1}{N}\sum_{i=1}^N \delta_{x_k^{(i)}}$. The full procedure is summarized in Table~\ref{tab:HB_particle}.

\noindent\textbf{Nesterov Acceleration Over Measures.} 
Analogously, Nesterov acceleration can be mapped into the probability measure space by evaluating the gradient field at an intermediate, look-ahead measure. We define the Nesterov drift and momentum transport maps by
\begin{align}
T_k^{\mathrm{N,drift}}(x, v) &:= \bigl(x + \sqrt{s}(1+\beta_k)v, \, v\bigr), \\
T_k^{\mathrm{N,momentum}}(x, v) &:= \left( x - \sqrt{s}\beta_k v, \, -\sqrt{s}\nabla_{W_2}F[\hat{\rho}_k](x) + \beta_k v \right),
\end{align}
where $\hat{\rho}_k := (T_k^{\mathrm{N,drift}})_{\#}\mu_{k-1}\bigr|^X$ acts as the look-ahead spatial measure. The resulting continuous measure iteration takes the form
\begin{equation}\label{eqn:N_sc}
\begin{cases}
\hat{\mu}_k = (T_k^{\mathrm{N,drift}})_{\#}\mu_{k-1}, \\[0.3em]
\mu_k = (T_k^{\mathrm{N,momentum}})_{\#}\hat{\mu}_k.
\end{cases}
\end{equation}

Deploying our particle discretization framework onto the continuous updates in \eqref{eqn:N_sc} yields the interacting Nesterov particle system:
\begin{equation}\label{eqn:Nesterov_particle}
\begin{cases}
y_k^{(i)} = x_{k-1}^{(i)} + \sqrt{s}(1+\beta_k)v_{k-1}^{(i)},\\[0.5em]
x_k^{(i)} = y_k^{(i)} - \sqrt{s}\beta_k v_{k-1}^{(i)}, \\[0.5em]
v_k^{(i)} = -\sqrt{s}\nabla_{W_2}F[\overline{\rho}_k^{\mathrm{look}}](y_k^{(i)}) + \beta_k v_{k-1}^{(i)},  \quad \text{with} \quad \overline{\rho}_k^{\mathrm{look}} = \frac{1}{N}\sum_{j=1}^N \delta_{y_k^{(j)}}.
\end{cases}
\end{equation}
Table~\ref{tab:Nes_particle_strong} summarizes the complete algorithmic layout for Nesterov acceleration.

\vspace{0.4cm}
\begin{table}[ht]
\centering
\caption{Heavy-ball algorithm for $m$-strongly geodesically convex functionals.}
\label{tab:HB_particle}
\vspace{0.2cm}
\small
\begin{tabular}{lll}
\toprule
 & \textbf{Continuous Measure-Level} & \textbf{Interacting Particle Approximation} \\ \midrule
\textbf{Parameters} & \multicolumn{2}{c}{$s = \frac{m}{16L^2}, \qquad \beta = \frac{1-\sqrt{ms}}{1+\sqrt{ms}}$} \\ \midrule
\textbf{Initialization} & Eq. \eqref{eqn:HB_initial} & Eq. \eqref{eqn:HB_initial_particle} \\ 
\textbf{Update Scheme} & Eq. \eqref{HB} & Eq. \eqref{eqn:particle_HB} \\ 
\textbf{Output Iterates} & $\rho_k = \mu_k^X$ & $\overline{\rho}_{k} = \frac{1}{N}\sum_{i=1}^N \delta_{x_k^{(i)}}$ \\ \bottomrule
\end{tabular}
\end{table}

\begin{table}[ht]
\centering
\caption{Nesterov acceleration algorithm for geodesically convex functionals (Nesterov-c and Nesterov-sc).}
\label{tab:Nes_particle_strong}
\vspace{0.2cm}
\small
\begin{tabular}{lll}
\toprule
 & \textbf{Continuous Measure-Level} & \textbf{Interacting Particle Approximation} \\ \midrule
\textbf{Parameters ($m=0$)} & \multicolumn{2}{c}{$s=\frac{1}{L}, \quad \kappa_0=1, \quad \kappa_{k+1}=\frac{1+\sqrt{1+4\kappa_{k}^2}}{2}, \quad \beta_k=\frac{\kappa_{k}-1}{\kappa_{k+1}}$} \\ \midrule
\textbf{Parameters ($m>0$)} & \multicolumn{2}{c}{$s=\frac{1}{L}, \qquad \beta_k = \beta = \frac{1-\sqrt{ms}}{1+\sqrt{ms}}$} \\ \midrule
\textbf{Initialization} & $\mu_0 = \rho_0 \otimes \delta_{0}$ & $x_0^{(i)} \overset{\mathrm{i.i.d.}}{\sim} \rho_0, \quad v_0^{(i)} = 0$ \\ 
\textbf{Update Scheme} & Eq. \eqref{eqn:N_sc} & Eq. \eqref{eqn:Nesterov_particle} \\ 
\textbf{Output Iterates} & $\rho_k = \mu_k^X$ & $\overline{\rho}_k = \frac{1}{N}\sum_{i=1}^N \delta_{x_k^{(i)}}$ \\ \bottomrule
\end{tabular}
\end{table}

\subsection{Convergence Analysis}\label{sec:convergence_theorem}

We analyze the convergence properties of the accelerated methods introduced in Section~\ref{sec:algorithms}. The analysis proceeds in two distinct stages. We first establish the convergence behavior of the continuous measure-level phase-space iterations \eqref{HB} and \eqref{eqn:N_sc}. This tracks convergence rates with respect to the iteration count $k$, verifying that the accelerated rates of classical Euclidean schemes are preserved when translated to probability space.

We then analyze the full particle approximations \eqref{eqn:particle_HB} and \eqref{eqn:Nesterov_particle} in Section~\ref{sec:particle_convergence}, explicitly quantifying the additional discrepancy introduced by the spatial Monte Carlo finite-sample discretization. Combining these elements provides end-to-end, non-asymptotic convergence guarantees tracking errors simultaneously across iteration count $k$, step size $\sqrt{s}$, and total particle allocation $N$.

Throughout this analysis, the Hilbert-space lifting developed in Section~\ref{sec:lions_diff} acts as the central engine. It restores a uniform, flat linear geometry where convexity and Lipschitz smoothness can be parsed through standard functional analytic techniques. This allows accelerated complexity estimates to be cleanly mapped from Euclidean domains directly to probability measure spaces.

\subsubsection{Convergence of Continuous Phase-Space Dynamics}

We first examine the convergence behavior of the continuous measure-level phase-space iterations \eqref{HB} and \eqref{eqn:N_sc}, viewed as optimization dynamics over the space of probability measures. Under the structural conditions of Assumption~\ref{assum:regularity}, the continuous heavy-ball and Nesterov schemes achieve non-asymptotic convergence rates that mirror their classical Euclidean counterparts.

\begin{theorem}[Continuous Measure-Level Heavy-Ball convergence]\label{HB_theorem}
Let $F:\mathcal{P}_2(\mathbb{R}^d) \to \mathbb{R}$ be an $m$-strongly geodesically convex functional for some $m > 0$ that satisfies Assumption~\ref{assum:regularity} and the gradient regularity condition~\eqref{eqn:F_diff_L_smooth}. Let $\{\rho_k\}_{k\ge 0} \subset \mathcal{P}_2(\mathbb{R}^d)$ be the sequence of position marginals generated by the continuous measure-level heavy-ball flow outlined in Table~\ref{tab:HB_particle}. Then the iterates satisfy:
\begin{equation}
F[\rho_k] - {m_F} \le \mathcal{O}\!\left(\left(1 + \frac{m}{16L}\right)^{-k}\right).
\end{equation}
\end{theorem}

\begin{theorem}[Continuous Measure-Level Nesterov convergence]\label{Nes_theorem}
Let $F:\mathcal{P}_2(\mathbb{R}^d) \to \mathbb{R}$ satisfy Assumption~\ref{assum:regularity} and the gradient regularity condition~\eqref{eqn:F_diff_L_smooth}. Let $\{\rho_k\}_{k\ge 0} \subset \mathcal{P}_2(\mathbb{R}^d)$ be the sequence of position marginals generated by the continuous measure-level Nesterov flow outlined in Table~\ref{tab:Nes_particle_strong}. Then the following rates hold:
\begin{enumerate}[leftmargin=3em]
    \item[\normalfont\bfseries{(i)}] If $F$ is $m$-strongly geodesically convex for some $m > 0$, the iterates exhibit linear convergence:
    \begin{equation}
    F[\rho_k] - {m_F} \le \mathcal{O}\!\left(\left(1 - \sqrt{\frac{m}{L}}\right)^k\right).
    \end{equation}
    \item[\normalfont\bfseries{(ii)}] If $F$ is geodesically convex $(m = 0)$, the iterates exhibit accelerated sublinear convergence:
    \begin{equation}
    F[\rho_k] - {m_F} \le \mathcal{O}\!\left(\frac{1}{k^2}\right).
    \end{equation}
\end{enumerate}
\end{theorem}

The key engine of the proof is to represent the joint phase-space distribution $\mu_k$ via a lifted variable $T_k \in L^2(\nu) \times L^2(\nu)$ such that $(T_k)_{\#}\nu = \mu_k$. This allows us to translate the nonlinear Wasserstein dynamics into a linear iteration in the Hilbert space $L^2(\nu)$. By Proposition~\ref{prop2.2} and Proposition~\ref{L2 lip cond}, the structural convexity and smoothness constants of $F$ map identically to the lifted objective $\widetilde{F}$, enabling us to carry out the derivation via dimension-free optimization analysis~\citep{su2016differential, nesterov2013introductory}. We present the proof of Theorem~\ref{Nes_theorem}(i) as a canonical representation.

\begin{proof}[Proof of Theorem~\ref{Nes_theorem}(i)]
Recall that the continuous Nesterov measure update is governed by \eqref{eqn:N_sc} with the initial configuration $\mu_0 = \rho_0 \otimes \delta_{0}$. Let $\mathsf{P}_x$ and $\mathsf{P}_v$ denote the canonical projection operators onto the position and velocity coordinates, respectively. We construct the lifted iteration operators acting on the joint Hilbert space $L^2(\nu) \times L^2(\nu)$ as:
\begin{equation}\label{eqn:SR_operator}
S := \begin{pmatrix} \mathsf{P}_x + \sqrt{s}(1+\beta)\,\mathsf{P}_v \\ \mathsf{P}_v \end{pmatrix}, \qquad R := \begin{pmatrix} \mathsf{P}_x - \sqrt{s}\beta\,\mathsf{P}_v \\ -\sqrt{s}\,\nabla_{L^2}\widetilde{F} \circ \mathsf{P}_x + \beta\,\mathsf{P}_v \end{pmatrix}.
\end{equation}
Then, for each discrete step $k = 1, 2, \dots$, the lifted state sequence updates via:
\begin{equation}\label{T_update}
T_k = R \circ S \circ T_{k-1}, \quad \text{with intermediate maps } \hat{T}_k = S \circ T_{k-1} \text{ and } T_k = R \circ \hat{T}_k.
\end{equation}
To initialize, let $X_0 \in L^2(\nu)$ satisfy $(X_0)_{\#}\nu = \rho_0$, and define $T_0 := (X_0, 0)$, which trivially ensures $(T_0)_{\#}\nu = \mu_0$. We establish that $(T_k)_{\#}\nu = \mu_k$ for all $k \ge 1$ by induction. Assuming $(T_{k-1})_{\#}\nu = \mu_{k-1}$ holds, we have:
\begin{equation}
(\hat{T}_k)_{\#}\nu = S_{\#}\bigl((T_{k-1})_{\#}\nu\bigr) = S_{\#}\mu_{k-1} = \hat{\mu}_k,
\end{equation}
by applying the pushforward composition property \eqref{pushforward}. By identical logic:
\begin{equation}
(T_k)_{\#}\nu = R_{\#}\bigl((\hat{T}_k)_{\#}\nu\bigr) = R_{\#}\hat{\mu}_k = \mu_k.
\end{equation}
Now, extracting the coordinate variables from the joint Hilbert operators, we define:
\begin{equation}
X_k := \mathsf{P}_x T_k, \qquad Y_k := \mathsf{P}_x \hat{T}_k, \qquad V_k := \mathsf{P}_v T_k.
\end{equation}
Given that $(T_k)_{\#}\nu = \mu_k$, it follows directly that:
\begin{equation}
(X_k)_{\#}\nu = (\mathsf{P}_x T_k)_{\#}\nu = (\mathsf{P}_x)_{\#}\bigl((T_k)_{\#}\nu\bigr) = \mu_k^X = \rho_k.
\end{equation}
Thus, the operational layout in \eqref{T_update} can be explicitly expanded component-wise as:
\begin{equation}\label{eqn:lifted_nesterov_component}
\begin{cases}
Y_k = X_{k-1} + \sqrt{s}(1+\beta)V_{k-1}, \\[0.3em]
V_k = -\sqrt{s}\,\nabla_{L^2}\widetilde{F}(Y_k) + \beta\,V_{k-1}, \\[0.3em]
X_k = Y_k - \sqrt{s}\beta\,V_{k-1},
\end{cases}
\end{equation}
where we have leveraged the gradient identity $\nabla_{L^2}\widetilde{F}(Y_k) = \nabla_{W_2}F(\hat{\rho}_k) \circ Y_k$ from \eqref{identity}. 

This system possesses the exact algorithmic structure of the classical Euclidean Nesterov sequence \eqref{eqn:ecl_Nes_yvx}, mapped over the infinite-dimensional Hilbert space $L^2(\nu)$. Because the standard convergence analysis for Nesterov's scheme relies entirely on inner-product properties and is fundamentally dimension-free, it extends directly to the $L^2(\nu)$ setting~\citep{nesterov2013introductory}. Specifically, since $\widetilde{F}$ is verified to be $m$-strongly convex and $L$-smooth via Propositions~\ref{prop2.2} and \ref{L2 lip cond}, classical convergence results yield:
\begin{equation}
\widetilde{F}(X_k) - \widetilde{F}(X^\star) \le \left(1 - \sqrt{\frac{m}{L}}\right)^k \frac{L+m}{2}\|X_0 - X^\star\|_\nu^2,
\end{equation}
where $X^\star = \mathrm{arg\,min}_{X \in L^2(\nu)} \widetilde{F}(X)$. 

Let $\rho^\star := (X^\star)_{\#}\nu$. By the construction of the Hilbert lift, $\widetilde{F}(X^\star) = F[\rho^\star]$. Since $X^\star$ minimizes $\widetilde{F}$ over $L^2(\nu)$, its pushforward $\rho^\star$ minimizes $F$ over $\mathcal{P}_2(\mathbb{R}^d)$, forcing $\widetilde{F}(X^\star) = F[\rho^\star] = {m_F}$. 

Furthermore, because this relation holds for any initial representative variable $X_0$ validating $(X_0)_{\#}\nu = \rho_0$, we can explicitly select $X_0$ such that the joint distribution $(X_0, X^\star)_{\#}\nu$ forms the unique optimal coupling between $\rho_0$ and $\rho^\star$. This choice guarantees that $\|X_0 - X^\star\|_\nu^2 = W_2^2(\rho_0, \rho^\star)$. Thus, we conclude:
\begin{equation}
F[\rho_k] - {m_F} = \widetilde{F}(X_k) - {m_F} \le \left(1 - \sqrt{\frac{m}{L}}\right)^k \frac{L+m}{2}W_2^2(\rho_0, \rho^\star).
\end{equation}
\end{proof}

\subsubsection{Convergence of the Monte Carlo Representation}\label{sec:particle_convergence}

We now shift to a more practical computational setting where optimization trajectories are evolved via the finite-sample interacting particle updates derived in~\eqref{eqn:particle_HB} and~\eqref{eqn:Nesterov_particle}. Taking Nesterov acceleration as our primary analytical vehicle, the following theorem quantifies the explicit non-asymptotic error accumulation induced by the spatial Monte Carlo discretization, demonstrating that the sample error propagates gracefully alongside the accelerated iteration count.

\begin{theorem}[Particle Discretization Convergence for Nesterov Acceleration]\label{thm:particle_convergence}
Assume that the functional $F: \mathcal{P}_2(\mathbb{R}^d) \to \mathbb{R}$ satisfies Assumption~\ref{assum:regularity} along with the gradient smoothness condition~\eqref{eqn:F_diff_L_smooth}. Assume further that $M_{3}(\mu_0)<+\infty$ (finite third moment). Let $\{\mu_k\}_{k\ge 0}$ be the exact continuous measure-level trajectory generated by the Nesterov update scheme~\eqref{eqn:N_sc}, and let $\{\overline{\mu}_k^N\}_{k\ge 0}$ be its interacting finite-particle approximation tracking~\eqref{eqn:Nesterov_particle}. Let $\overline{\rho}_0^N := \frac{1}{N}\sum_{i=1}^N \delta_{x_0^{(i)}}$ denote the initial empirical spatial allocation. Then, for any fixed optimization time horizon $t = k\sqrt{s}$, there exists a dimension-dependent positive constant $C > 0$ such that the expected phase-space Wasserstein error satisfies:
\begin{equation}
\mathbb{E}\left[ W_2\left(\mu_{\lfloor t/\sqrt{s} \rfloor}, \, \overline{\mu}_{\lfloor t/\sqrt{s} \rfloor}^N\right) \right] \le C e^{(3+L)t} N^{-1/d}, \qquad \text{for } d > 4.
\end{equation}
\end{theorem}

Before diving into the algebraic details, we can understand the core idea of the proof through a simple geometric property. The algorithm updates both the continuous measure~\eqref{eqn:N_sc} and the empirical particle system~\eqref{eqn:Nesterov_particle} using the exact same pair of operators, $S$ and $R$. Crucially, these operators act directly on the coordinate fields of the lifted space and do not depend on the underlying measure $\rho$ itself. Because both trajectories are pushed forward by the same functional mechanisms, comparing them simply boils down to tracking the stability. By establishing that the operators are Lipschitz continuous, we guarantee that any initial sampling error from the Monte Carlo initialization stayed small over time.

\begin{proof}
We select initial joint configurations $T_0, \overline{T}_0^N \in L^2(\nu; \mathbb{R}^d \times \mathbb{R}^d)$ such that they validate the pushforward relations:
\begin{equation}
(T_0)_{\#}\nu = \mu_0, \qquad (\overline{T}_0^N)_{\#}\nu = \overline{\rho}_0^N \otimes \delta_{0}.
\end{equation}
Recall the lifted continuous operator composition from \eqref{T_update}, where $\mathsf{P}_x$ and $\mathsf{P}_v$ denote the canonical position and velocity projection operators, respectively. We utilize the joint mapping operators $S$ and $R$ defined over the Hilbert space as in~\eqref{eqn:SR_operator}, to express updates of the measures:
\begin{equation}
T_k = R \circ S \circ T_{k-1}, \qquad \overline{T}_k^N = R \circ S \circ \overline{T}_{k-1}^N.
\end{equation}

Since the momentum inertia parameter satisfies $\beta = \frac{\sqrt{L}-\sqrt{m}}{\sqrt{L}+\sqrt{m}} < 1$, a direct evaluation shows that the drift operator $S$ is $(1 + 2\sqrt{s})$-Lipschitz continuous. Concurrently, because the lifted functional $\widetilde{F}$ is globally $L$-smooth on the unified Hilbert space via Proposition~\ref{L2 lip cond}, the acceleration adjustment operator $R$ behaves as a $(1 + \sqrt{s}(L+1))$-Lipschitz map. Consequently, the operational composition $R \circ S$ is $L'$-Lipschitz continuous, with its combined constant bounded by:
\begin{equation}
L' := (1 + 2\sqrt{s})\bigl(1 + \sqrt{s}(L+1)\bigr) = 1 + \sqrt{s}(3+L) + 2s(L+1).
\end{equation}
Evaluating the discrepancy between the true continuous trajectory and the empirical particle sequence at step $k+1$ inside the Hilbert space $L^2(\nu)$, we find:
\begin{equation}
\|T_{k+1} - \overline{T}_{k+1}^N\|_\nu = \|(R \circ S \circ T_k) - (R \circ S \circ \overline{T}_k^N)\|_\nu \le L' \|T_k - \overline{T}_k^N\|_\nu.
\end{equation}
Applying this structural relation recursively back to the initial state yields:
\begin{equation}\label{eqn:recursive_operator_bound}
\|T_k - \overline{T}_k^N\|_\nu \le (L')^k \|T_0 - \overline{T}_0^N\|_\nu.
\end{equation}

By choosing the initial lifted representatives $T_0$ and $\overline{T}_0^N$ such that the joint coupling $((T_0)_{\#}\nu, (\overline{T}_0^N)_{\#}\nu)$ is optimal on the product space, we obtain:
\begin{equation}
W_2(\mu_k, \, \overline{\mu}_k^N) \le \|T_k - \overline{T}_k^N\|_\nu \le (L')^k W_2(\mu_0, \, \overline{\mu}_0^N).
\end{equation}
Taking expectations across both sides under the sampling distribution gives:
\begin{equation}\label{eqn:expectation_fournier}
\mathbb{E}\left[ W_2(\mu_k, \, \overline{\mu}_k^N) \right] \le (L')^k \mathbb{E}\left[ W_2(\mu_0, \, \overline{\mu}_0^N) \right] \le C (L')^k N^{-1/d},
\end{equation}
where the final inequality is a direct consequence of standard empirical measure concentration bounds on continuous spaces~\citep{fournier2015rate} and $M_{3}(\mu_0)<+\infty$ is used.

Finally, parameterizing this bound via the continuous optimization time scale $t = k\sqrt{s}$, we can rewrite the iteration step exponent as $k = \lfloor t/\sqrt{s} \rfloor$, yielding:
\begin{equation}
\mathbb{E}\left[ W_2\left(\mu_{\lfloor t/\sqrt{s} \rfloor}, \, \overline{\mu}_{\lfloor t/\sqrt{s} \rfloor}^N\right) \right] \le C (L')^{t/\sqrt{s}} N^{-1/d} \le C e^{(3+L)t} N^{-1/d},
\end{equation}
where the final exponential scaling follows directly from the standard analytic comparison $e^x \ge 1 + x + \frac{x^2}{2}$, evaluating the exponent at $x = (3+L)\sqrt{s}$. 
\end{proof}

%% file: numerical_experiment.tex
\section{Numerical experiments}\label{sec:numerics}
In this section, we apply the heavy-ball algorithm (Table~\ref{tab:HB_particle}) and the Nesterov algorithms (Tables~\ref{tab:Nes_particle_strong}) to three representative tasks. We adopt the same experimental settings as in \citep{CLTW25}, and implement all methods through particle approximations. Example~1 considers the minimization of potential energy functionals. Example~2 studies the minimization of a regularized KL divergence toward a prescribed target distribution, a problem arising in Bayesian sampling. Example~3 concerns the training of an infinitely wide two-layer neural network with ReLU activation. Example~4 concerns the solution to the Porous media dynamics that formulates as a nonlinear Fokker--Planck equation.

\bigskip
\noindent\textbf{Example 1: Potential Energy.} Consider the potential energy function
\[F[\rho]=\mathcal V_l(\rho)=\int_{\mathbb{R}^d} V_l(x) d\rho(x),\qquad l=1,2,\]
with two different choice of potential functions: \[
V_1(x)=\frac{1}{2}\langle x-b,A(x-b)\rangle;\qquad V_2(x)=h\log\left(\sum_{i=1}^M\exp\left(\frac{\langle w_i,x\rangle-q_i}{h}\right)\right)+\frac{\delta}{2}\|x\|_2^2.
\] 
\begin{figure}[ht]
    \centering
    \includegraphics[width=0.48\textwidth]{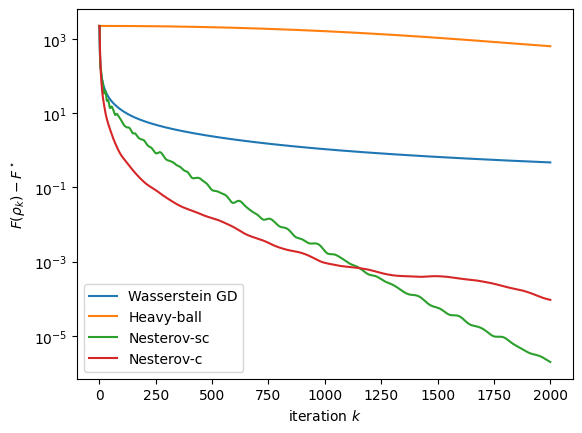}
    \hfill
    \includegraphics[width=0.48\textwidth]{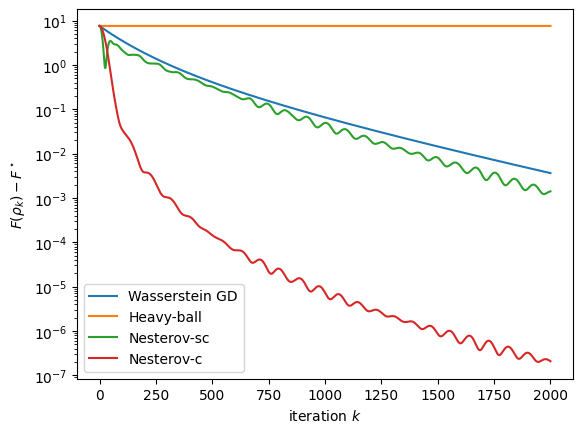}
    \caption{Optimality gap \(F(\rho_k)-F^*\) versus iteration \(k\) for the two test potential energy functionals, comparing Wasserstein gradient descent, heavy-ball, Nesterov-sc, and Nesterov-c. The left panel corresponds to the quadratic potential \(\mathcal V_1\), while the right panel corresponds to the ridge neural-network potential \(\mathcal V_2\).}
    \label{fig:potential}
\end{figure}

For the potential \(\mathcal V_1\), we set the dimension to \(d=500\), and choose a random symmetric positive definite matrix \(A\in\mathbb R^{d\times d}\) whose eigenvalues are drawn independently from a log-uniform distribution on \([10^{-5},1]\). Consequently, \(\mathcal V_1\) is \(m\)-strongly geodesically convex with \(m\approx 10^{-5}\), and its lifting is \(L\)-smooth with \(L\approx 1\). The vector \(b\in\mathbb R^d\) is sampled from \(\mathcal N(0,100\,I_d)\). For the potential \(\mathcal V_2\), we take \(d=200\), \(h=20\), \(\delta=10^{-5}\), and \(M=1000\). The parameters \(w_i\in\mathbb R^d\) are drawn independently from \(\mathcal N(0,I_d)\), and \(q_i\in\mathbb R\) are drawn independently from \(\mathcal N(0,1)\). Under this setting, \(\mathcal V_2\) is convex, with only weak curvature contributed by the ridge term \(\delta\|x\|_2^2/2\). We estimate its smoothness parameter by \(L \approx \frac{1}{h}\max_i \|w_i\|_2^2+\delta\). In both experiments, we use \(N=100\) particles to represent the underlying distribution.

Figure~\ref{fig:potential} illustrates the performance of Wasserstein gradient descent, heavy-ball, and Nesterov with two different $(s,\beta)$ choices on the two test potentials \(\mathcal V_1\) and \(\mathcal V_2\). In the strongly convex setting \(\mathcal V_1\), both Nesterov variants exhibit clear acceleration, with Nesterov-sc (with $\beta$ independent of $k$) achieving the fastest convergence. In the nearly weakly convex setting \(\mathcal V_2\), Nesterov-c provides the most significant acceleration. By contrast, in both experiments, the heavy-ball method performs substantially worse than gradient descent. We also note that error plots for Hamiltonian-based methods usually demonstrate oscillations, as expected for Hamiltonian systems.

\bigskip
\noindent\textbf{Example 2: Bayesian Sampling.} Our second task is to minimize the KL divergence between \(\rho\) and a target distribution \(\rho^\star\), namely 
\[
F[\rho]=\mathsf{KL}(\rho\|\rho_l^\star)
=\int_{\mathbb R^d}\rho(x)\ln\frac{\rho(x)}{\rho_l^\star(x)}\,dx,
\qquad l=1,2,
\]
where the target distribution \(\rho_l^\star\) admits a density of the form \(\rho_l^\star\propto e^{-V_l(x)}\). This task has two main difficulties. It is important to stress that, as discussed in Section~\ref{sec:background}, the KL divergence does not satisfy Assumption~\ref{assum:regularity}. Moreover, the dynamics generated by the updates may fail to be differentiable, even in the case of Wasserstein gradient descent; see \citep{xu2025forwardeulerwassersteingradient}. Consequently, our main convergence theorem does not apply in this setting. We therefore consider instead the regularized KL divergence
\begin{equation}\label{eqn:KL_epsilon}
F^\varepsilon(\rho)=\mathsf{KL}^\varepsilon(\rho\|\rho_l^\star)
:=\int_{\mathbb R^d}\ln\bigl(K^\varepsilon * \rho\bigr)\,d\rho
+\int_{\mathbb R^d} V_l(x)\,d\rho
-\ln C,
\end{equation}
where \(K^\varepsilon * \rho\) denotes the convolution of \(\rho\) with the Gaussian kernel
$K^\varepsilon=(\frac{1}{2\pi\varepsilon^2})^{d/2}\exp(-|x|^2/2\varepsilon^2)$.
As illustrated in \citep{carrillo2019blob}, if \(\rho_l^\star\) is \(m\)-strongly convex with \(m\) sufficiently large, then \(F^\varepsilon\) is convex and bounded from below. Moreover, as shown in \citep{xu2025forwardeulerwassersteingradient}, the lifting \(\widetilde F^\varepsilon\) it is smooth on \(\mathcal P_2(\mathcal C)\) whenever \(\mathcal C\subset \mathbb R^d\) is a bounded convex set.

For visualization, we first consider a two dimensional case with
\[V(x_1,x_2)
=
\frac12\left(
\frac{(\cos(0.55)x_1+\sin(0.55)x_2)^2}{0.55^2}
+
\frac{(-\sin(0.55)x_1+\cos(0.55)x_2)^2}{0.28^2}
\right).
\]
Iterations are ran up to $k=1000$ using $N=1600$, and set \(\varepsilon=0.04\). In Figure~\ref{fig:KL_2D_error} we plot the error decay in iteration for all four methods. The two errors are respectively, objective functional~\eqref{eqn:KL_epsilon}, and $W_2$ distance between the computed measure (represented by particles) and the desired density. It is clear that both variants of Nesterov perform better. In Figure~\ref{fig:KL_2D_visual} we visualize the density by four methods at five different iterations ($k=0$, $250$, $500$, $750$ and  $1000$). All four methods start with a Gaussian mixture initial that is far from the desired target distribution. Only the two variants of Nesterov manage to reconstruct the target distribution.
\begin{figure}[ht]
    \centering
    \includegraphics[width=1\linewidth]{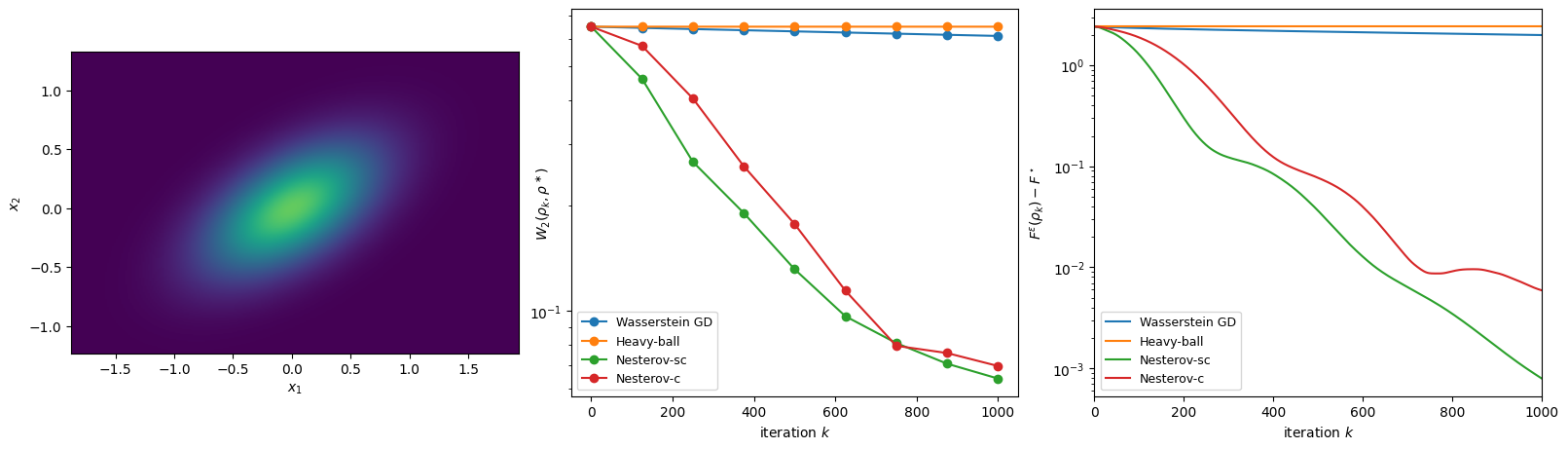}
    \caption{Panel (a) shows the target distribution. Panel (b) and (c) plot the error and objective function decay in iteration.}\label{fig:KL_2D_error}
\end{figure}
\begin{figure}[ht]
    \centering
    \includegraphics[width=.8\linewidth]{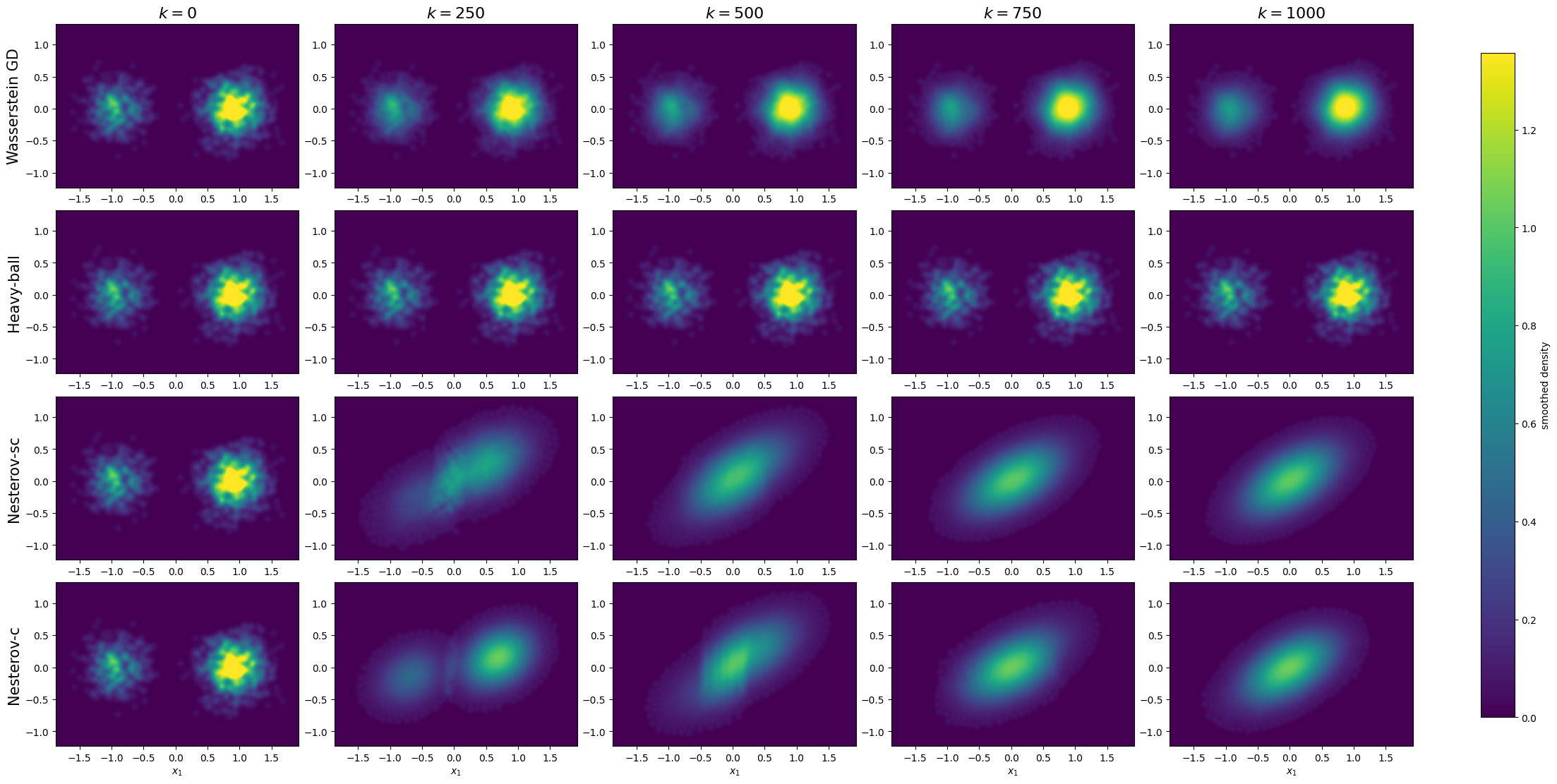}
    \caption{Heat map for evolution of $\rho$ by four different methods at five iterations. To plot the density, we smooth out the particle presentation with a Gaussian smoothing kernel $K^\varepsilon$.}\label{fig:KL_2D_visual}
\end{figure}

The procedure can be repeated in high dimension space as well. We consider the same two choices of potential as in Example~1. We let $\rho_l^\star$ to have the log-density:
\[
V_1(x)=\frac{1}{2}\langle x-b,A(x-b)\rangle,
\qquad
V_2(x)=h\log\left(\sum_{i=1}^M\exp\left(\frac{\langle w_i,x\rangle-q_i}{h}\right)\right)+\frac{\delta}{2}\|x\|_2^2\,,
\]
where for the log-density \(V_1\), we take \(d=20\) and choose a random symmetric positive definite matrix \(A\in\mathbb R^{d\times d}\) whose eigenvalues are drawn independently from a log-uniform distribution on \([10^{-4},1]\). The vector \(b\in\mathbb R^d\) is sampled from \(\mathcal N(0,10\,I_d)\). For \(V_2\), we set \(d=10\), \(h=10\), \(M=200\), and \(\delta=10^{-5}\). The parameters \(w_i\in\mathbb R^d\) are drawn independently from \(\mathcal N(0,I_d)\), and \(q_i\in\mathbb R\) are drawn independently from \(\mathcal N(0,1)\). In both experiments, we choose \(\varepsilon=0.1\) and use \(N=1600\) particles. 

{To set parameters for $s$ and $\beta_k$,} we use the convexity parameter of \(V_l\) as an estimator of the strong convexity constant \(m\), and use an empirical local Lipschitz estimate of the particle gradient as the smoothness parameter \(L\). To approximate the optimal value \(F^*\), we run all four methods for a sufficiently long time and take the best objective value attained among these runs as our estimate of \(F^*\).

Figure~\ref{fig:KL} presents the results for minimizing the two regularized KL divergences associated with log-density \(V_1\) and \(V_2\). It shows a similar pattern to Figure~\ref{fig:potential}. In both experiments, the heavy-ball method does not provide good performance. By contrast, both Nesterov variants consistently outperform Wasserstein gradient descent. For $V_1$, the potential is quadratic, Nesterov-sc decreases the objective most rapidly at the beginning, while Nesterov-c attains the smallest final gap. For $V_2$, both Nesterov schemes again improve substantially over gradient descent, with Nesterov-sc showing a slight advantage in the long run.
\begin{figure}[ht]
    \centering
        \includegraphics[width=0.48\textwidth]{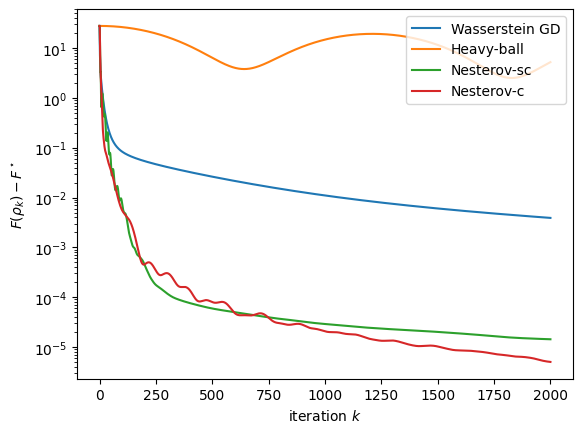}
    \hfill
        \includegraphics[width=0.48\textwidth]{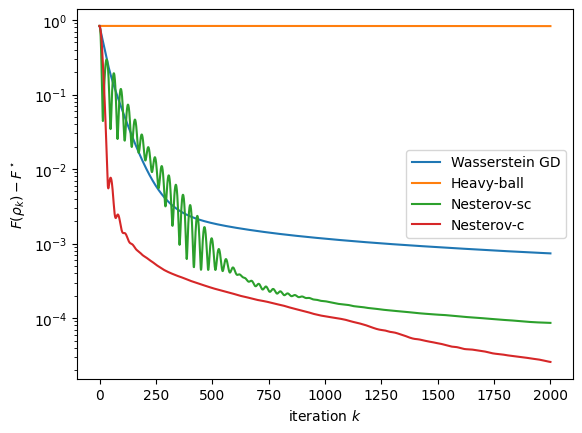}
    \caption{Optimality gap \(F^\varepsilon(\rho_k)-F^*\) versus iteration \(k\) for Wasserstein gradient descent, heavy-ball, Nesterov-sc, and Nesterov-c applied to the two regularized KL divergences. The two panels show results corresponding to the quadratic log-density \(V_1\), and \(V_2\) respectively.}
    \label{fig:KL}
\end{figure}

\bigskip
\noindent\textbf{Example 3: Neural network training.}
Finally, we consider a more challenging task: training infinitely wide neural networks \citep{DCLW22}. Specifically, we minimize the functional
\begin{equation}\label{eqn:objective_NN_function}
F[\rho]=\frac{1}{2}\int_{\mathbb R^d} \bigl|f(x)-g(x,\rho)\bigr|^2\,d\pi(x),
\end{equation}
where \(\pi\) is a given data distribution and \(f:\mathbb R^d\to\mathbb R\) is the target function. We take \(g\) to be a two-layer neural network: for \(x\in\mathbb R^d\) and \(\rho\in\mathcal P_2(\mathbb R^{d+3})\),
\[
g(x,\rho):=\int_{\mathbb R^{d+3}} V(x,z)\,d\rho(z),
\qquad
V(x,(\alpha,w,b))=\alpha\,\sigma(w\cdot x+b),
\]
where \(z=(\alpha,w,b)\in \mathbb R\times\mathbb R^d\times\mathbb R\), and \(\sigma(x)=\max(0,x)\) is the ReLU activation. Although \(F\) may fail to achieve globally geodesically convex, it is locally geodesically convex, as outlined in \citep{CLTW25}.

In the experiment, we take \(d=1\) and choose the target function \(f(x)=\sin(\pi x)\). We sample 1500 data points uniformly from \([-1,1]\) to approximate the integral with respect to \(\pi\), and use \(N=400\) particles. As shown in Figure~\ref{fig:NN}, both Nesterov variants significantly outperform Wasserstein gradient descent in this task, whereas the heavy-ball method again performs poorly.
\begin{figure}[ht]
    \centering
    \includegraphics[width=0.48\textwidth]{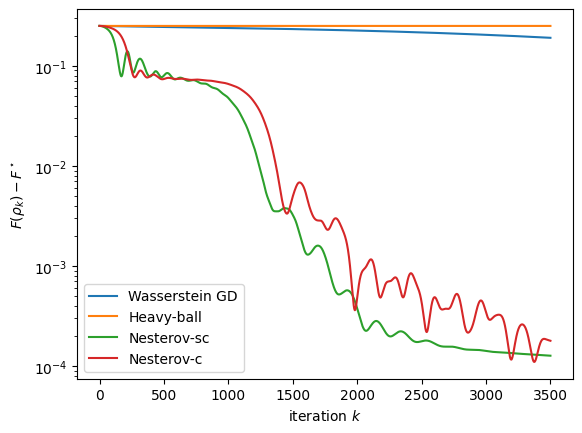}
    \hfill
    \includegraphics[width=0.48\textwidth]{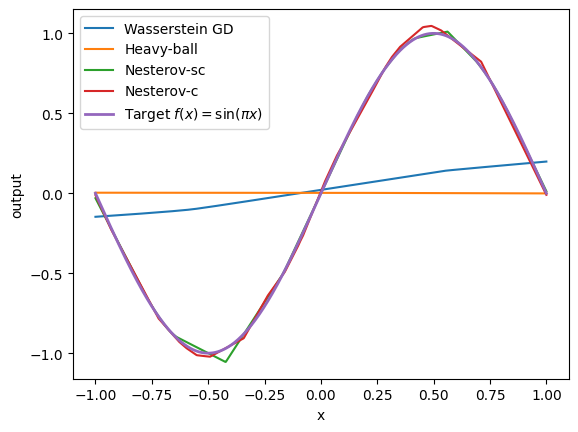}

    \caption{Performance of neural network training with target
    \(f(x)=\sin(\pi x)\) for Wasserstein gradient descent, heavy-ball,
    Nesterov-SC, and Nesterov-C. The left panel shows the training loss
    \(F(\rho_k)\) versus the iteration index \(k\). The right panel compares
    the target function with the neural network approximations produced by
    the four methods after \(k=3500\) iterations.}
    \label{fig:NN}
\end{figure}

An extension to 2-dimensional function-approximation provides similar result. We set the ground-truth function to be:
\begin{align*}
    f(x_1,x_2)
&=
0.9 \exp\!\left(-18\bigl((x_1-0.32)^2 + (x_2-0.34)^2\bigr)\right)
+
0.75 \exp\!\left(-16\bigl((x_1-0.70)^2 + (x_2-0.68)^2\bigr)\right)\\
&-
0.55 \exp\!\left(-10\bigl((x_1-0.50)^2 + (x_2-0.50)^2\bigr)\right)
+
0.12 \sin(2\pi x_1)\sin(2\pi x_2)\,,
\end{align*}
and learn the underlying distribution $\rho$ that represents the neural network by optimizing~\eqref{eqn:objective_NN_function}. The ground-truth function is plotted on panel (a) in Figure~\ref{fig:objective_NN_function2}.

In Figure~\ref{fig:objective_NN_function2} panel (b) we show the decay of objective function using four different methods. Both Nesterov variants provide faster convergence in iteration. Figure~\ref{fig:NN_function_2} further shows the learned function at five different iterations ($k=0$, $4000$, $8000$, $12000$ and $16000$) by four different methods.
\begin{figure}[ht]
    \centering
    \includegraphics[width=.82\linewidth]{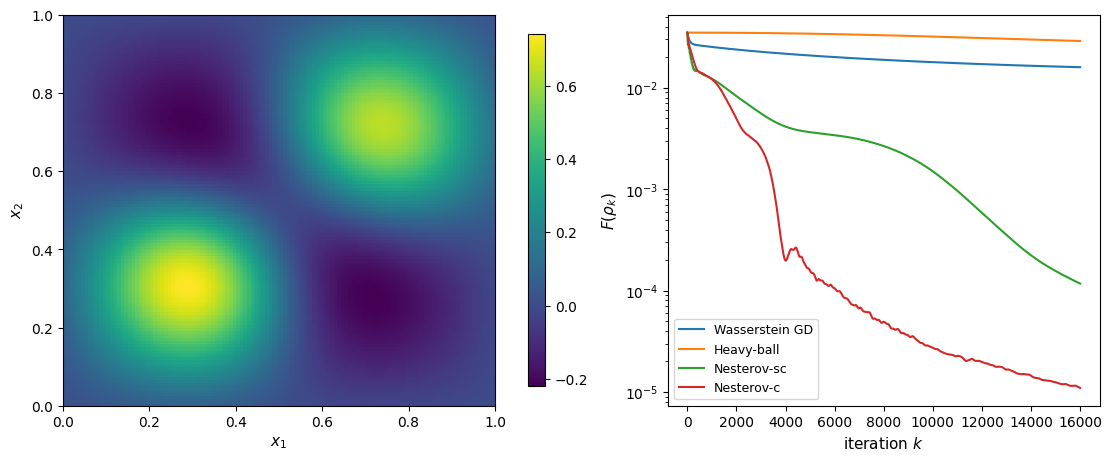}
        \caption{The panel on the left shows the heat map of the ground-truth function. The panel on the right shows decay of objective functional in iterations.}\label{fig:objective_NN_function2}
\end{figure}

\begin{figure}[ht]
    \centering
    \includegraphics[width=.7\linewidth]{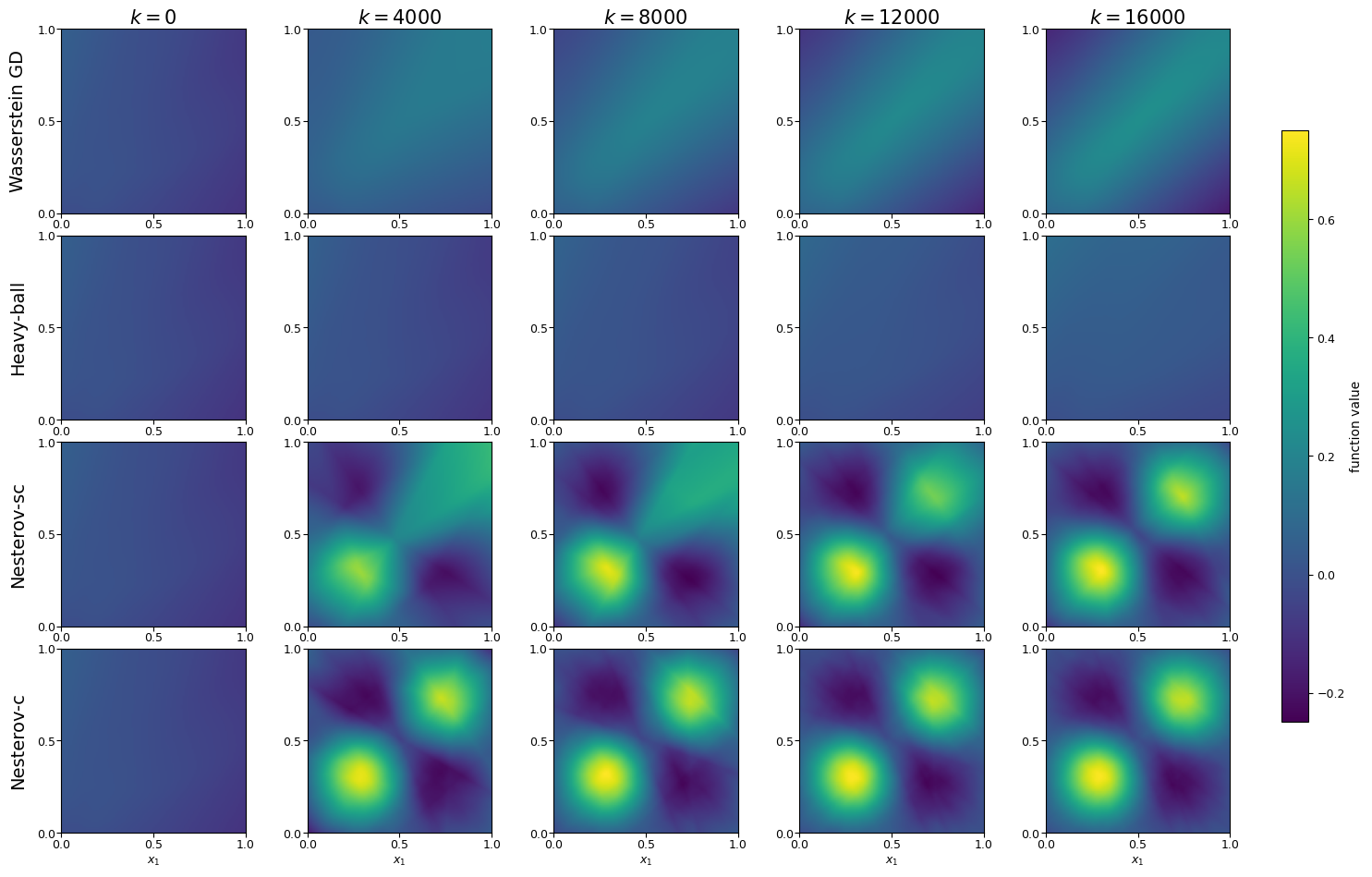}
    \caption{Heat map for evolution of reconstructed functions by four different methods at five iterations. }\label{fig:NN_function_2}
\end{figure}

\bigskip
\noindent\textbf{Example 4: Porous media and nonlinear Fokker--Planck.}
Besides these canonical examples, energy minimization can give rise to nonlinear Fokker-Planck equation as well~\citep{carrillo2019blob}. In particular, the functional
\[
    F(\rho)=\int \frac{|x|^2}{2}\,d\rho+\int \rho^2\,dx
\] 
is designed over the density space, and its associated gradient flow reads:
\[
\partial_t \rho
=
\nabla\!\cdot(\rho \nabla V)+\Delta(\rho^2),
\qquad
V(x)=\frac{|x|^2}{2}\,.
\]
This formulation requires $\rho$ to be absolutely continuous, but a smoothed version is a good approximation and extends to all probability measures:
\[ F^\varepsilon(\rho)
=
\int \frac{|x|^2}{2}\,d\rho(x)
+
\int (K^\varepsilon * \rho)(x)\,d\rho(x)\,,
\]
In 2-dimension, the target of the minimizer for origin functional $F$ is
\[
\psi_2(0.25,x)
:=
2\left(\frac{1}{\sqrt{8\pi}}-\frac{|x|^2}{8}\right)_+\,.
\]
We set $N=3000$ and $\varepsilon=0.04$ and run all four methods to look for the minimizer of the objective function. In Figure~\ref{fig:nFP_error} we plot the known optimal distribution in panel (a), together with the decay of error, measured in Wasserstein sense, and the objective function, in iteration in panel (b) and (c). In Figure~\ref{fig:nFP_heat_visual} we plot the solution provided by four different methods along iterations ($k=0$, $500$, $1000$, $1500$ and $2000$). It is clear that both variants of Nesterov converge faster and Nesterov-sc provides the best reconstruction of the desired distribution.
\begin{figure}[h]
    \centering
    \includegraphics[width=1\linewidth]{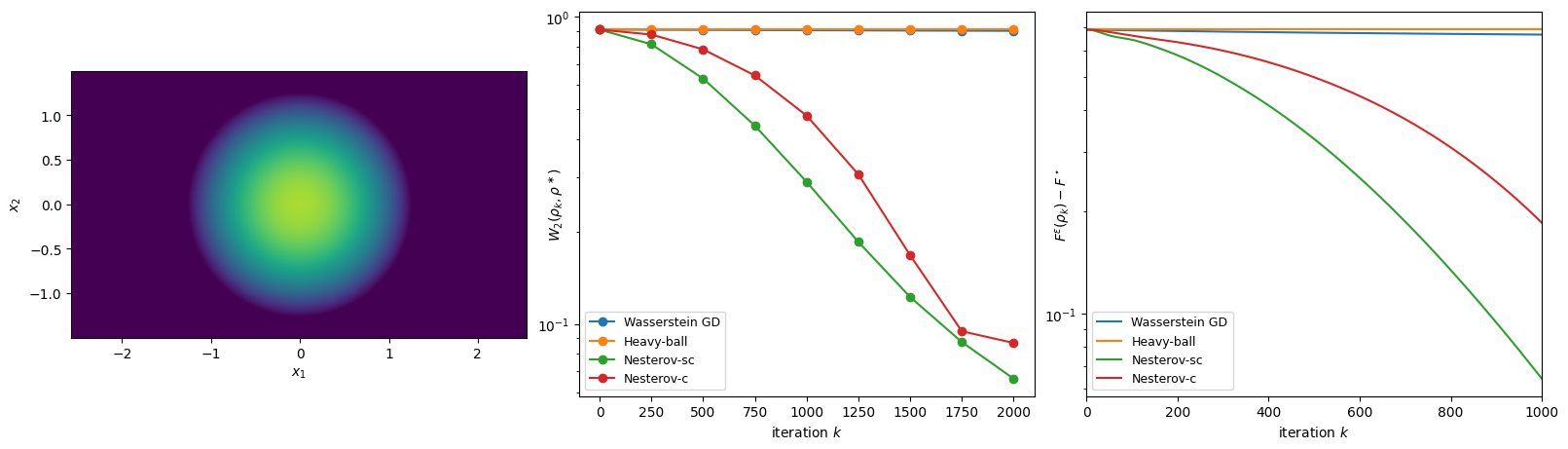}
    \caption{The left panel shows the desired distribution $\psi_2(0.25,x)$. The middle and right panel shows error and objective function decay in iteration.}\label{fig:nFP_error}
\end{figure}
\begin{figure}[h]
    \centering
    \includegraphics[width=.8\linewidth]{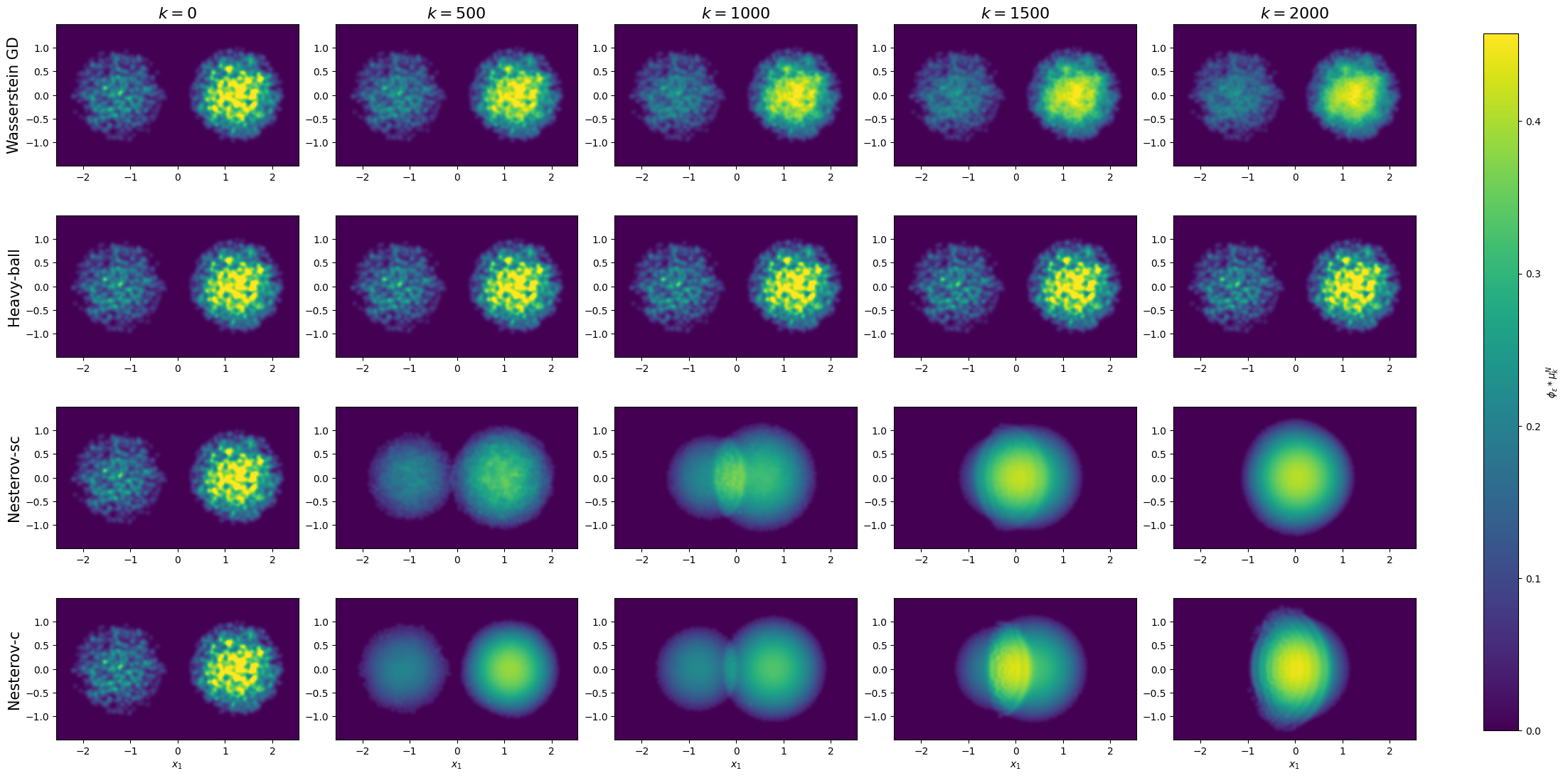}
    \caption{Heat map for evolution of $\rho$ given by four different methods at five different iteration. To plot the density, we smooth out the particle presentation with a Gaussian smoothing kernel $K^\varepsilon$.}\label{fig:nFP_heat_visual}
\end{figure}